\documentclass[11pt,a4paper]{article}
\usepackage[utf8]{inputenc}
\usepackage{amsmath}
\usepackage{amsfonts}
\usepackage{amssymb}
\usepackage{amsthm}
\usepackage{makeidx}
\usepackage{graphicx}
\usepackage{authblk}
\usepackage{subfig}
\usepackage{subeqnarray}
\usepackage[T1]{fontenc}
\usepackage[top=2cm, bottom=2cm, left=2cm, right=2cm]{geometry}
\usepackage{marginnote}

\usepackage{mdframed}
\usepackage{imakeidx}
\usepackage{hyperref}
\usepackage{color}
\usepackage[pagewise]{lineno}

\title{The $p$-Laplacian in thin channels with \\ locally periodic rough boundaries}

\newtheorem{theorem}{Theorem}[section]
\newtheorem{corollary}{Corollary}[theorem]
\newtheorem{lemma}[theorem]{Lemma}
\newtheorem{definition}[theorem]{Definition}
\newtheorem{proposition}[theorem]{Proposition}
\newtheorem{remark}{Remark}[section]
\numberwithin{equation}{section}
\author[1]{Jean Carlos Nakasato\thanks{e-mail: nakasato@ime.usp.br. Partially supported by CNPq scolarship (Brazil).}}
\author[1]{Marcone Corr\^ea Pereira\thanks{e-mail: marcone@ime.usp.br. Partially supported by CNPq 303253/2017-7 and FAPESP 2019/06221-5 (Brazil).}}
\affil[1]{Depto. Matem\'atica Aplicada, IME, Universidade de S\~ao Paulo,
	Rua do Mat\~ao 1010, S\~ao Paulo - SP, Brazil}

\begin{document}
	\maketitle
	\begin{abstract}
		In this work we analyze the asymptotic behavior of the solutions of the $p$-Laplacian equation with homogeneous Neumann boundary conditions set in bounded thin domains as 
		$$R^\varepsilon=\left\lbrace(x,y)\in\mathbb{R}^2:x\in(0,1)\mbox{ and }0<y<\varepsilon G\left(x,{x}/{\varepsilon}\right)\right\rbrace.$$
		
		We take a smooth function $G:(0,1)\times\mathbb{R} \mapsto \mathbb{R}$, $L$-periodic in the second variable, which allows us to consider locally periodic oscillations at the upper boundary. The thin domain situation is established passing to the limit in the solutions as the positive parameter $\varepsilon$ goes to zero. 
	\end{abstract}

	
	\section{Introduction}
		
	Let $G$ be a function 
	
	\begin{mdframed}[frametitlefont=\bfseries\large,frametitle={(H)},frametitlealignment=\centering]
		$G:(0,1)\times\mathbb{R} \mapsto \mathbb{R}$ satisfying that there exist a finite number of points 
		$$0=\xi_0<\xi_1<\cdots<\xi_{N-1}<\xi_N=1$$ 
		such that $G:(\xi_{i-1},\xi_i)\times\mathbb{R}\to (0,\infty)$ is $C^1$ and such that $G$, $\partial_x G$ and 
		$\partial_y G$ are uniformly bounded in $(\xi_{i-1},\xi_i)\times\mathbb{R}$ getting limits when we approach $\xi_{i-1}$ and $\xi_i$.
		Further, we assume there exist two constants $G_0$ and $G_1$ such that 
		\begin{equation*}  \label{g0}
		0<G_0 \leq G(x,y)\leq G_1, \quad \forall (x,y) \in (0,1)\times\mathbb{R},
		\end{equation*}
		and a real number $L>0$ such that $G(x,y+L)=G(x,y)$ for all $(x,y)\in (0,1)\times\mathbb{R}$.\footnote{$G(x,\cdot)$ is a $L$-periodic function for each $x \in (0,1)$.}
	\end{mdframed}

	We denote by $R^\varepsilon$ the following family of open bounded sets 
	\begin{equation} \label{TDintro}
	R^\varepsilon=\left\lbrace(x,y)\in\mathbb{R}^2:x\in(0,1)\mbox{ and }0<y<\varepsilon G_\varepsilon\left(x\right)\right\rbrace \quad \textrm{ for } \varepsilon>0,
	\end{equation}
	where 
	$$
	G_\varepsilon(x)=G\left(x,\frac{x}{\varepsilon}\right).
	$$

	In this note, we are interested in analyzing the asymptotic behavior of the solutions of a 
	$p$-Laplacian equation posed in the thin domain $R^\varepsilon$ with rough boundary.
	We consider 
	\begin{equation}\label{problem}
	\left\lbrace \begin{gathered}
	-\Delta_p u_\varepsilon +|u_\varepsilon|^{p-2}u_\varepsilon=f^\varepsilon\mbox{ in }R^\varepsilon \\
	|\nabla u_\varepsilon|^{p-2}\nabla u_\varepsilon \eta_\varepsilon =0\mbox{ on }\partial R^\varepsilon
	\end{gathered}\right.
	\end{equation}
	where $\eta_\varepsilon$ is the unit outward normal vector to the boundary $\partial R^\varepsilon$, $1<p<\infty$ with $p^{-1}+p'^{-1}=1$, and 
	$$
	\Delta_p\cdot = \mbox{div }\left(|\nabla \cdot|^{p-2}\nabla\cdot\right)
	$$
	denotes the $p$-Laplacian differential operator. We also assume $f^\varepsilon\in L^{p'}(R^\varepsilon)$ uniformly bounded.

	The variational formulation of \eqref{problem} is given by
	\begin{equation}\label{variationalproblem}
	\int_{R^\varepsilon} \left\{ |\nabla u_\varepsilon|^{p-2}\nabla u_\varepsilon\nabla \varphi +|u_\varepsilon|^{p-2}u_\varepsilon\varphi \right\}dxdy=\int_{R^\varepsilon} f^\varepsilon\varphi \, dxdy, \qquad \varphi \in W^{1,p}(R^\varepsilon).
	\end{equation}
	Existence and uniqueness of the solutions are guaranteed by Minty-Browder's Theorem setting a family of solutions $u_\varepsilon$. 
	We study the asymptotic behavior of $u_\varepsilon$ as $\varepsilon \to 0$, as the domain $R^\varepsilon$ becomes thinner and thinner, also exhibiting a high oscillating boundary at the top due to $L$-periodicity of $G(x,\cdot)$ as illustrated at Figure \ref{fig1}.
	\begin{figure}[htp] 
\centering \scalebox{0.6}{\includegraphics{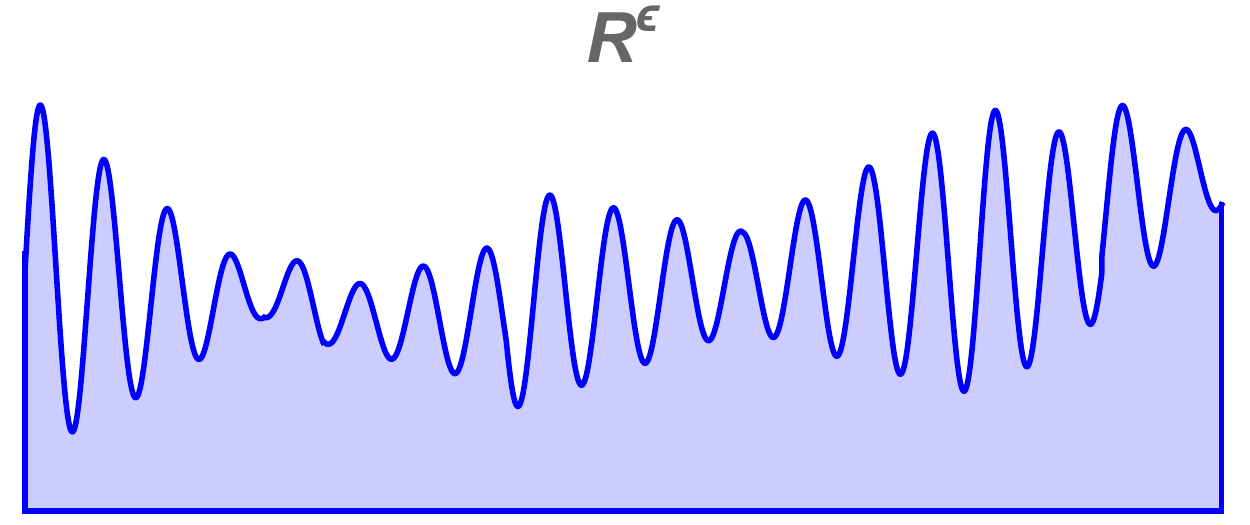}}
\caption{A locally periodic thin channel with rough boundary.}
\label{fig1} 
\end{figure}

	According to \cite{arrieta2011} and references there in, it is expected that the sequence $u_\varepsilon$ will converge to a function of just one variable $x \in (0,1)$ satisfying a one-dimensional equation of the same type.
	Combining boundary perturbation techniques \cite{AP10, AM,AM2} and monotone operator analysis \cite{lindqvist}, we identify the effective limit model of \eqref{variationalproblem} at $\varepsilon=0$. 	
	
	As we will see, the homogenized limit equation is a one-dimensional $p$-Laplacian equation with variable coefficients $q(x)$ and $r(x)$. It assumes the following form
	\begin{equation*} \label{limeq}
	\left\lbrace \begin{gathered}
	- \left(q(x) |u'|^{p-2}u'\right)' +r(x) |u|^{p-2}u=\hat{f} \quad \mbox{ in }(0,1),\\
	u'(0)=u'(1)=0,
	\end{gathered}\right. 
	\end{equation*}
	where
	\begin{equation}\label{coeq}
	\begin{gathered}
	q(x)=\dfrac{1}{L}\int_{Y^*(x)}|\nabla v|^{p-2}\partial_{y_1} v \, dy_1dy_2\\
	r(x)= \dfrac{|Y^*(x)|}{L}
	\end{gathered}
	\end{equation}
	and $|Y^*(x)|$ denotes the Lebesgue measure of the representative cell 
	$$
	Y^*(x)=\left\{(y_1,y_2):0<y_1<L,0<y_2<G(x,y_1) \right\}
	$$
	which also depends on variable $x \in (0,1)$. 
	The function $v$ used to set the homogenized coefficient $q(x)$ in \eqref{coeq} is an auxiliar function which is the unique solution of the problem  
	\begin{equation}\label{auxiliarproblemi}
	\begin{gathered}
	\displaystyle\int_{Y^*(x)}\left|\nabla v\right|^{p-2}\nabla v\nabla \varphi \, dy_1dy_2=0 \quad \forall\varphi\in W^{1,p}_{\#}(Y^*(x)), 
	\qquad \left\langle \varphi \right\rangle_{Y^*(x)} = 0, \\
	(v-y_1)\in  W^{1,p}_{\#}(Y^*(x)) \qquad \textrm{ with } \quad \left\langle (v-y_1) \right\rangle_{Y^*(x)} = 0,
	\end{gathered}
	\end{equation}
	where  
	$$
	W_{\#}^{1,p}(Y^*(x)) = \{ \varphi\in W^{1,p}(Y^*(x)) \, : \,  \varphi |_{\partial_{left} Y^*(x)} = \varphi|_{\partial_{right} Y^*(x)} \}
	$$
	is the space of periodic functions on the horizontal variable $y_{1}$, and 	$\left\langle \varphi\right\rangle_{\mathcal{O}}$ denotes the average of any function $\varphi \in L^1_{loc}(\mathbb{R}^N)$ on measurable sets $\mathcal{O} \subset \mathbb{R}^N$. 
	
	It is worth noticing that problem \eqref{auxiliarproblemi} is well posed for each $x\in(0,1)$, due to Minty-Browder's Theorem, and then, the coefficient $q(x)$ is also well defined. Further, $q(x)$ is a positive function setting a well posed homogenized equation. Indeed, since $v$ is the solution of \eqref{auxiliarproblemi}, there exists $\psi\in W_{\#}^{1,p}(Y^*(x)) $ with $\left\langle \psi\right\rangle_{Y^*(x)}=0$ for each $x \in (0,1)$ such that $v=y_{1}+\psi$ implying 
	$$
	\begin{gathered}
	0<\int_{Y^*(x)}|\nabla v|^{p}\, dy_1dy_2=\int_{Y^*(x)}|\nabla v|^{p-2}\nabla v \nabla (y_{1}+\psi)\, dy_1dy_2\\ 
	=\int_{Y^*(x)}|\nabla v|^{p-2}\partial_{y_1} v \, dy_1dy_2=L \, q(x).
	\end{gathered}
	$$
	
	Several are the works in the literature dealing with issues related to the effect of thickness and rough on the feature of the solutions of partial differential equations. 	
	Indeed, thin structures with oscillating boundaries appear in many fields of science: fluid dynamics (lubrication), solid mechanics (thin rods, plates or shells) or even physiology (blood circulation). Therefore, analyzing the asymptotic behavior of models set on thin structures understanding how the geometry and the roughness affect the problem is a very relevant issues in applied science. 
	In these direction, see for instance \cite{BPazanin,GH,Jimbo,AJM,MZAMP} and references therein. 
	
	Here, we are improving results from \cite{AP10, AM} where the Laplacian operator in locally periodic thin domains were considered dealing with the $p$-Laplacian equation for any $p \in (1, \infty)$. 
	It is worth noticing that the techniques developed in \cite{AP10, AM} can not be directly applied in or case. On one side, the results obtained in \cite{AM} do not guarantee strong convergence in $L^{p}$ for the unfolding operator applied on the solutions of the quasilinear operators. On the other side, the analysis performed in \cite{AP10}  just work on $L^2$-spaces. Our goal here is to overcome this situation. We discretize the oscillating region passing to the limit using uniform estimates on two parameters: one associated to the roughness, and other given by the variable profile of the thin domain. In this way, a continuous dependence property for the solutions with respect to $G$ in $L^p$-norms is crucial.

%
%
%
%
	
	The main result of the paper is the following.
	\begin{theorem}\label{mainthm}
		Let $u_{\varepsilon}$ be the solution of \eqref{problem} with $f^\varepsilon\in L^{p'}(R^\varepsilon)$ uniformly bounded. Suppose that 
		$$
		\hat{f}^\varepsilon(x)=\frac{1}{\varepsilon}\int_0^{\varepsilon G\left(x,\frac{x}{\varepsilon}\right)}f^\varepsilon(x,y)dy
		$$
		satisfies
		$
		\hat{f}^\varepsilon\rightharpoonup\hat{f}\mbox{ weakly in }L^{p'}(0,1).
		$
		
		Then, there exists $u\in W^{1,p}(0,1)$ such that
		$$
		\dfrac{L}{|Y^*(x)|\varepsilon}\int_0^{\varepsilon G\left(x,\frac{x}{\varepsilon}\right)}u_{\varepsilon}(x,y) dy \rightharpoonup u 
		\mbox{ weakly in }L^{p}\left(0,1\right), \quad \textrm{ as } \varepsilon \to 0, 
		$$
		with $u$ satisfying the homogenized equation
		$$
		\int_0^1 \left\lbrace q(x)|u'|^{p-2}u'\varphi' +r(x)\,|u|^{p-2}u\varphi \right\rbrace dx=\int_0^1\hat{f}\varphi dx, \quad \forall \varphi\in W^{1,p}(0,1),
		$$
		where the homogenized coefficients $q(x)$ and $r(x)$ are given by \eqref{coeq}.
	\end{theorem}

		Notice that our paper also goes a step beyond \cite{MRi2} where the $p$-Laplacian operator is studied in standard thin domains with no oscillatory boundary (as those ones introduced in \cite{HaleRaugel}), and the recent one \cite{nakasato1} where purely periodically thin domains in oscillating boundary has been considered. 

	The paper is organized as follows. In Section \ref{secnot}, we introduce some notations and state some basic results which will be needed in the sequel. In Section \ref{secdomdep}, we prove the continuous dependence of the solutions in $L^p$ spaces with respect to the function $G$ uniformly in $\varepsilon>0$. In Section \ref{secpiecewise}, we perform the asymptotic analysis of \eqref{problem} in piecewise periodic thin domains (that is, in thin domains set by functions $G$ which are piecewise constants in the first variable $x$, and $L$-periodic in the second one). See Figure \ref{figpw} below which illustrates piecewise periodic open sets.
	\begin{figure}[!h]
		\centering
		\includegraphics[width=.45\columnwidth]{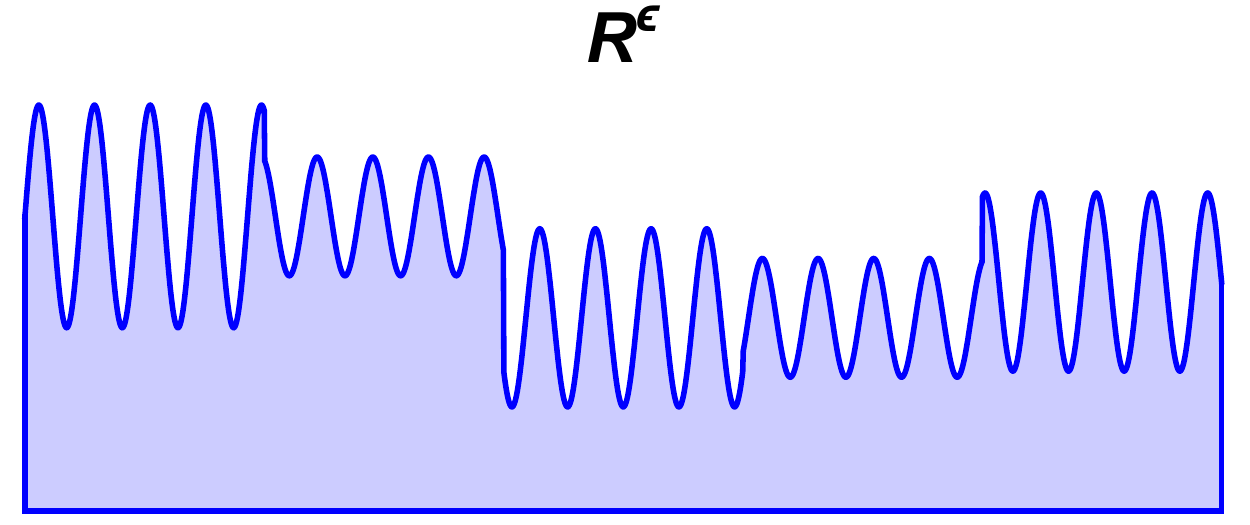} 
		\caption{A piecewise periodic thin domain.}
		\label{figpw} 
	\end{figure} 
	Finally, we provide a proof of the main result in Section \ref{secgencase}, namely Theorem \ref{mainthm}, as a consequence of the analysis performed in the previous sections. Furthermore, we include an Appendix where the dependence of the auxiliary solution $v$ on admissible functions $G$ is analysed.

	\section{Preliminaries}\label{secnot}
	
	In this section, we introduce some basic facts, definitions and results concerning to the unfolding method making some straightforward  adaptations to our propose. First, let us just recall some basic properties to the $p$-Laplacian which can be found for instance in \cite{lindqvist}. 
	\begin{proposition}\label{proposicaoplaplaciano}
		Let $x,y\in\mathbb{R}^n$.
		\begin{itemize}
			\item If $p\geq 2$, then
			\begin{equation*}\label{eq1}
			\langle|x|^{p-2}x-|y|^{p-2}y,x-y\rangle \geq c_p|x-y|^p.
			\end{equation*}
			\item If $1<p<2$, then
			\begin{eqnarray*}\label{eq2}
				\langle|x|^{p-2}x-|y|^{p-2}y,x-y\rangle&\geq&c_p|x-y|^2(|x|+|y|)^{p-2}\nonumber\\ &\geq& c_p|x-y|^2(1+|x|+|y|)^{p-2}.
			\end{eqnarray*}
		\end{itemize}
	\end{proposition}

	\begin{corollary}\label{corolarioplaplaciano}
		Let $a_{p}:\mathbb{R}^n\rightarrow \mathbb{R}^n$ such that $a_{p}(s)=|s|^{p-2}s$, $\frac{1}{p}+\frac{1}{p'}=1$. Then, $a_{p}$ is the inverse of $a_{p'}$. Moreover,
		\begin{itemize}
			\item If $1<p'< 2$ (i.e, $p\geq 2$), then
			\begin{equation*}
			\left||u|^{p'-2}u-|v|^{p'-2}v\right|\leq c|u-v|^{p'-1}.
			\end{equation*}
			\item If $p'\geq 2$ (i.e, $1<p<2$), then
			\begin{eqnarray*}
				\left||u|^{p'-2}u-|v|^{p'-2}v\right|&\leq& c|u-v|(|u|+|v|)^{p'-2}\nonumber\\
				&\leq& c|u-v|(1+|u|+|v|)^{p'-2}.
			\end{eqnarray*}
		\end{itemize}
	\end{corollary}
	
	\begin{proposition}\label{proposicaoplaplacian01}
		Let $x,y\in\mathbb{R}^n$ and $p\geq1$. Then.
		\begin{equation*}
		|y|^p\geq |x|^p+p|x|^{p-2}x\cdot(y-x)
		\end{equation*}
		Moreover, 
		\begin{eqnarray*}
			&&|y|^p\geq |x|^p+p|x|^{p-2}x\cdot(y-x)+c_p|y-x|^p\mbox{ if }p\geq 2,\\
			&&|y|^p\geq |x|^p+p|x|^{p-2}x\cdot(y-x)+c_p|x-y|^2(1+|x|+|y|)^{p-2}\mbox{ if }1<p<2.
		\end{eqnarray*}
	\end{proposition}
	
	From now on, we use the following rescaled norms 
	\begin{equation*}\label{normslp}
		\begin{gathered}
		\left|\left|\left|\varphi\right|\right|\right|_{L^p(R^\varepsilon)}=\varepsilon^{-1/p}\left|\left|\varphi\right|\right|_{L^p(R^\varepsilon)} \forall\varphi\in L^p(R^\varepsilon), 1\leq p< \infty,\\
		\left|\left|\left|\varphi\right|\right|\right|_{W^{1,p}(R^\varepsilon)}=\varepsilon^{-1/p}\left|\left|\varphi\right|\right|_{W^{1,p}(R^\varepsilon)}\forall\varphi\in W^{1,p}(R^\varepsilon), 1\leq p< \infty.
		\end{gathered}
	\end{equation*} 
	For completeness we may denote $\left|\left|\left|\varphi\right|\right|\right|_{L^\infty(R^\varepsilon)}=\left|\left|\varphi\right|\right|_{L^\infty(R^\varepsilon)}$. 
	
	Next, we get the following uniform bound for the solutions of \eqref{problem}:
	\begin{proposition}\label{uniformlimtation}
		Consider the variational formulation of our problem:
		\begin{equation}\label{1001}
		\int_{R^\varepsilon} \left\{ \left|\nabla u_\varepsilon\right|^{p-2}\nabla u_\varepsilon\nabla\varphi+\left|u_\varepsilon\right|^{p-2} u_\varepsilon\varphi \right\} dxdy = \int_{R^\varepsilon}f^\varepsilon\varphi dxdy, \; \varphi\in W^{1,p}(R^\varepsilon),
		\end{equation}
		where $f^\varepsilon$ satisfies $$\left|\left|\left|f^\varepsilon\right|\right|\right|_{L^{p'}(R^\varepsilon)}\leq c$$ for some positive constant $c$ independent of $\varepsilon>0$.
		Then, 
		\begin{equation*}
		\begin{gathered}
		\left|\left|\left|u_\varepsilon\right|\right|\right|_{L^{p}(R^\varepsilon)}\leq c,\quad \quad 
		\left|\left|\left|u_\varepsilon\right|\right|\right|_{W^{1,p}(R^\varepsilon)}\leq c,\\
		\left|\left|\left|\left|\nabla u_\varepsilon\right|^{p-2}\nabla u_\varepsilon\right|\right|\right|_{L^{p'}(R^\varepsilon)}\leq c.
		\end{gathered}
		\end{equation*}
	\end{proposition}
	\begin{proof}
		Take $\varphi=u_\varepsilon$ in \eqref{1001}. Then,
		\begin{eqnarray*}
			\left|\left|u_\varepsilon\right|\right|_{W^{1,p}(R^\varepsilon)}^p  = 
			\int_{R^\varepsilon} \left\{ \left|\nabla u_\varepsilon\right|^p+\left| u_\varepsilon\right|^p \right\} dx dy
			\leq \left|\left|f^\varepsilon\right|\right|_{L^{p'}(R^\varepsilon)} \left|\left|u_\varepsilon\right|\right|_{L^p(R^\varepsilon)}.
		\end{eqnarray*}
		%
		%
		Hence, 
		\begin{equation*}
		\left|\left|\left|u_\varepsilon\right|\right|\right|_{W^{1,p}(R^\varepsilon)}\leq c.
		\end{equation*}
		
		Therefore, the sequence $u_\varepsilon$ and $|\nabla u_\varepsilon|^{p-2}\nabla u_\varepsilon$, are respectively bounded in $L^p(R^\varepsilon)$ and $(L^{p'}(R^\varepsilon))^2$ under the norm $\left|\left|\left|\cdot\right|\right|\right|$.
	\end{proof}

	\subsection{Unfolding operator}\label{sectionunfolding}
	
	Here, we present the unfolding operator for thin domains in the purely and locally periodic setting. They have been introduced in \cite{AM,AM2} where details an proofs can be found.

	\subsubsection{Purely periodic} \label{subsectionpurelyperiodic}
	
	Let $ G_i:\mathbb{R}\to\mathbb{R}$ be a $L$-periodic function, lower semicontinuous satisfying $0<g_{0,i} \leq G_i(x)\leq g_{1,i}$ with $g_{0,i} = \min_{x \in \mathbb{R}} G_i(x)$ and $g_{1,i} = \sup_{x \in \mathbb{R}} G_i(x)$ for any $i=1, ..., N$. Now consider the thin region 
	$$
	R^\varepsilon_i=\left\lbrace (x,y)\in\mathbb{R}: \xi_{i-1}<x<\xi_i,0<y<\varepsilon G_i(x/\varepsilon)\right\rbrace.
	$$
	The basic cell associated to $R^\varepsilon_i$ is
	\begin{equation*} 
	Y^*_{i}=\left\lbrace \right(y_1,y_2) \in \mathbb{R}^2 : 0<y_1<L\mbox{ and }0<y_2<G_{i}(y_1)\rbrace.
	\end{equation*}
	
	By 
	$$
	\left\langle \varphi\right\rangle_{\mathcal{O}} := \frac{1}{|\mathcal{O}|} \int_{\mathcal{O}} \varphi(x) \, dx,
	$$
	we denote the average of $\varphi \in L^1_{loc}(\mathbb{R}^2)$ for any open measurable set $\mathcal{O} \subset \mathbb{R}^2$. 
	We also set functional spaces which are defined by periodic functions in the variable $y_1 \in (0,L)$. Namely 
	$$
	\begin{gathered}
	L^p_\#(Y^*_{i}) = \{ \varphi \in L^p(Y^*_{i}) \, : \, \varphi(y_1,y_2) \textrm{ is $L$-periodic in $y_1$ } \}, \\
	L^p_\#\left((0,1)\times Y^*_{i}\right) =  \{ \varphi \in L^p((0,1) \times Y^*_{i}) \, : \, \varphi(x, y_1,y_2) \textrm{ is $L$-periodic in $y_1$ } \}, \\
	W_{\#}^{1,p}(Y^*_{i}) = \{ \varphi\in W^{1,p}(Y^*_{i}) \, : \,  \varphi |_{\partial_{left} Y^*_{i}} = \varphi|_{\partial_{right} Y^*_{i}}\}.
	\end{gathered}
	$$
	
	For each $\varepsilon>0$ and any $x\in (\xi_{i-1},\xi_i)$, there exists an integer denoted by $\left[\frac{x}{\varepsilon}\right]_L$ such that
	\begin{equation*}\label{notationpartition}
	x=\varepsilon  \left[\frac{x}{\varepsilon}\right]_LL+\varepsilon \left\{\frac{x}{\varepsilon }\right\}_L\mbox{ where }\left\{\frac{x}{\varepsilon }\right\}_L\in [0,L).
	\end{equation*}
	We still set
	\begin{equation*}
	I_\varepsilon^{i}=\mbox{ Int }\left(\bigcup_{k=0}^{N_\varepsilon^i}\left[kL\varepsilon ,(k+1)L\varepsilon \right]\right), 
	\end{equation*}
	where $N_\varepsilon^{i}$ is largest integer such that $\varepsilon  L(N_\varepsilon^{i}+1)\leq \xi_{i}$, as well
	$$
	\begin{gathered}
	\Lambda_\varepsilon^{i}=(\xi_{i-1},\xi_{i})\backslash I_\varepsilon^{i}=[\varepsilon  L(N_\varepsilon+1),\xi_{i}), \\
	R^\varepsilon_{0i}=\left\lbrace(x,y)\in \mathbb{R}^2:x\in I_\varepsilon^{i}, 0<y<\varepsilon  G_{i}\left(\frac{x}{\varepsilon}\right)\right\rbrace, \\
	R^\varepsilon_{1i}=\left\lbrace(x,y)\in \mathbb{R}^2:x\in \Lambda_\varepsilon^{i}, 0<y<\varepsilon  G_{i}\left(\frac{x}{\varepsilon}\right)\right\rbrace.
	\end{gathered}
	$$
	
	\begin{definition}
		Let $\varphi$ be a Lebesgue-measurable function in $R^\varepsilon_i$. The unfolding operator $\mathcal{T}_\varepsilon^i$ acting on $\varphi$ is defined as the following function in $(\xi_{i-1},\xi_i)\times Y^*_i$
		\begin{eqnarray*}
			\mathcal{T}^{i}_\varepsilon\varphi(x,y_1,y_2)=\left\lbrace\begin{array}{ll}
				\varphi\left(\varepsilon \left[\frac{x}{\varepsilon}\right]_LL+\varepsilon y_1,\varepsilon y_2\right)\mbox{, for }(x,y_1,y_2)\in I_\varepsilon^i\times Y^*_i,\\
				0, \mbox{ for }(x,y_1,y_2)\in \Lambda_\varepsilon^i\times Y^*_i.
			\end{array}\right.
		\end{eqnarray*}
	\end{definition}
	\begin{proposition}\label{unfoldingproperties}
		The unfolding operator satifies the following properties:
		\begin{enumerate}
			\item $\mathcal{T}_\varepsilon^i$ is linear;
			\item $\mathcal{T}_\varepsilon^i(\varphi\psi)=\mathcal{T}_\varepsilon^i(\varphi)\mathcal{T}_\varepsilon(\psi)$, for all $\varphi$, $\psi$ Lebesgue mesurable in $R^\varepsilon_i$;
			\item $\forall\varphi\in L^p(R^\varepsilon_i)$, $1\leq p\leq \infty$,
			\begin{equation*}
			\mathcal{T}^{i}_\varepsilon(\varphi)\left(x,\left\lbrace\frac{x}{\varepsilon}\right\rbrace_L,\dfrac{y}{\varepsilon}\right)=\varphi(x,y),
			\end{equation*}
			for $(x,y)\in R_{0i}^\varepsilon$.
			\item Let $\varphi$ a Lebesgue mesurable function in $Y^*_i$ extended periodically in the first variable. Then, $\varphi^\varepsilon(x,y)=\varphi\left(\dfrac{x}{\varepsilon},\dfrac{y}{\varepsilon}\right)$ is mesurable in $R^\varepsilon_i$ and
			\begin{equation*}
			\mathcal{T}_\varepsilon^i(\varphi^\varepsilon)(x,y_1,y_2)=\varphi(y_1,y_2),\forall(x,y_1,y_2)\in I_\varepsilon\times Y^*_i.
			\end{equation*} 
			Moreover, if $\varphi\in L^p(Y^*_i)$, then $\varphi^\varepsilon\in L^p(R^\varepsilon_i)$;
			\item If  $f:Y^*_i\times\mathbb{R}^2\to\mathbb{R}^2$ is $L$-periodic in $y_1$ and $\varphi:R^\varepsilon_i\to \mathbb{R}^2$ a mesurable function, then $$\mathcal{T}_\varepsilon^i \left[f\left(\dfrac{\cdot}{\varepsilon},\dfrac{\cdot}{\varepsilon},\varphi\right)\right]=f\left(y_1,y_2, \mathcal{T}_\varepsilon^i\varphi\right)$$
			for $(x,y_1,y_2)\in I_\varepsilon\times Y^*_i$;
			\item Let $\varphi^\varepsilon\in L^1(R^\varepsilon_i)$. Then,
			\begin{eqnarray*}
				& & \frac{1}{L}\int_{(\xi_{i-1},\xi_i)\times Y^*_i}\mathcal{T}_\varepsilon^i(\varphi)(x,y_1,y_2)dxdy_1dy_2=\dfrac{1}{\varepsilon}\int_{R_{0i}^\varepsilon}\varphi(x,y)dxdy\\
				&& \qquad \qquad = \dfrac{1}{\varepsilon}\int_{R^\varepsilon_i}\varphi(x,y)dxdy-\dfrac{1}{\varepsilon} \int_{R_{1i}^\varepsilon}\varphi(x,y)dxdy;
			\end{eqnarray*}
			\item $\forall \varphi \in L^p(R^\varepsilon_i)$, $\mathcal{T}_\varepsilon^i(\varphi)\in L^p\left((\xi_{i-1},\xi_i)\times Y^*_i\right)$, $1\leq p\leq \infty$. Moreover
			\begin{equation*}
			\left|\left|\mathcal{T}_\varepsilon^i(\varphi)\right|\right|_{L^p\left((\xi_{i-1},\xi_i)\times Y^*_i\right)}=\left(\dfrac{L}{\varepsilon}\right)^{\frac{1}{p}}\left|\left|\varphi\right|\right|_{L^p(R_{0i}^\varepsilon)}\leq \left(\dfrac{L}{\varepsilon}\right)^{\frac{1}{p}}\left|\left|\varphi\right|\right|_{L^p(R^\varepsilon_i)}.
			\end{equation*}
			If $p=\infty$,
			\begin{equation*}
			\left|\left|\mathcal{T}_\varepsilon^i(\varphi)\right|\right|_{L^\infty\left((\xi_{i-1},\xi_i)\times Y^*_i\right)}=\left|\left|\varphi\right|\right|_{L^\infty(R_{0i}^\varepsilon)}\leq \left|\left|\varphi\right|\right|_{L^\infty(R_i^\varepsilon)};
			\end{equation*}
			\item $\forall\varphi \in W^{1,p}(R^\varepsilon_i)$, $1\leq p\leq \infty$,
			\begin{equation*}
			\partial_{y_1}\mathcal{T}_\varepsilon^i(\varphi)=\varepsilon \mathcal{T}_\varepsilon^i(\partial_{x}\varphi)\mbox{ e }\partial_{y_2}\mathcal{T}_\varepsilon^i(\varphi)=\varepsilon \mathcal{T}_\varepsilon^i(\partial_{y}\varphi)\mbox{ a.e. in }(\xi_{i-1},\xi_i)\times Y^*_i;
			\end{equation*}
			\item If $\varphi \in W^{1,p}(R^\varepsilon_i)$, then $\mathcal{T}_\varepsilon^i(\varphi)  \in L^p\left((\xi_{i-1},\xi_i);W^{1,p}(Y^*_i)\right)$, $1\leq p\leq \infty$. Besides, for $1\leq p< \infty$, we have 
			$$
			\begin{gathered}
			\left|\left|\partial_{y_1}\mathcal{T}_\varepsilon^i(\varphi)\right|\right|_{L^p\left((\xi_{i-1},\xi_i)\times Y^*_i\right)}=\varepsilon \left(\dfrac{L}{\varepsilon}\right)^{\frac{1}{p}}\left|\left|\partial_{x}\varphi\right|\right|_{L^p(R_{0i}^\varepsilon)}\leq \varepsilon \left(\dfrac{L}{\varepsilon}\right)^{\frac{1}{p}}\left|\left|\partial_{x}\varphi\right|\right|_{L^p(R_i^\varepsilon)}\\
			\left|\left|\partial_{y_2}\mathcal{T}_\varepsilon^i(\varphi)\right|\right|_{L^p\left((\xi_{i-1},\xi_i)\times Y^*_i\right)}=\varepsilon\left(\dfrac{L}{\varepsilon}\right)^{\frac{1}{p}}\left|\left|\partial_{y}\varphi\right|\right|_{L^p(R_{0i}^\varepsilon)}\leq \varepsilon\left(\dfrac{L}{\varepsilon}\right)^{\frac{1}{p}}\left|\left|\partial_{y}\varphi\right|\right|_{L^p(R_i^\varepsilon)}.
			\end{gathered}
			$$
			If $p=\infty$,
			\begin{eqnarray*}
				&&\left|\left|\partial_{y_1}\mathcal{T}_\varepsilon^i(\varphi)\right|\right|_{L^\infty\left((\xi_{i-1},\xi_i)\times Y^*_i\right)}=\varepsilon\left|\left|\partial_{x}\varphi\right|\right|_{L^\infty(R_{0i}^\varepsilon)}\leq \varepsilon\left|\left|\partial_{x}\varphi\right|\right|_{L^\infty(R^\varepsilon_i)}\\
				&&\left|\left|\partial_{y_2}\mathcal{T}_\varepsilon^i(\varphi)\right|\right|_{L^\infty\left((\xi_{i-1},\xi_i)\times Y^*_i\right)}=\varepsilon\left|\left|\partial_{y}\varphi\right|\right|_{L^\infty(R_{0i}^\varepsilon)}\leq \varepsilon\left|\left|\partial_{y}\varphi\right|\right|_{L^\infty(R_i^\varepsilon)}.
			\end{eqnarray*}
			\item Let $\left(\varphi_\varepsilon\right)$ be a sequence in $L^p(R_i^\varepsilon)$, $1<p\leq\infty$ with the norm $\left|\left|\varphi_\varepsilon\right|\right|_{L^p(R_i^\varepsilon)}$ uniformly bounded. 
			Then,
			\begin{equation*}
			\dfrac{1}{\varepsilon}\int_{R_{1i}^\varepsilon}|\varphi_\varepsilon|dxdy\rightarrow 0.
			\end{equation*}
			\item Let $(\varphi_\varepsilon)$ be a sequence in $L^p(\xi_{i-1},\xi_i)$, $1\leq p<\infty$, such that 
			\begin{equation*}
			\varphi_\varepsilon\rightarrow \varphi\mbox{ strongly in }L^p(\xi_{i-1},\xi_i).
			\end{equation*}
			Then, 
			\begin{equation*}\label{449}
			\mathcal{T}_\varepsilon^i\varphi_\varepsilon\rightarrow \varphi\mbox{ strongly in }L^p\left((\xi_{i-1},\xi_i)\times Y^*_i\right).
			\end{equation*}
		\end{enumerate}
	\end{proposition}
	
	The above result sets several basic and somehow immediate properties of the unfolding operator.
	%
	Property 6 of Proposition \ref{unfoldingproperties} will be essential to pass to the limit when dealing with solutions of differential equations since it allow us to transform any integral over the thin sets depending on the parameter $\varepsilon$ and  function $G_i$ into an integral over the fixed set $(\xi_{i-1},\xi_i) \times Y^*_i$. 
	
	\begin{remark} \label{re483} Since $|\cdot|^{p-2}\cdot$ is monotone, we have that
		$\mathcal{T}_\varepsilon^i f^\varepsilon\rightarrow f$ strongly in $L^p\left((\xi_{i-1},\xi_i)\times Y^*_i\right)$ implies  
		\begin{equation*}
		\mathcal{T}_\varepsilon^i \left(|f^\varepsilon|^{p-2}f^\varepsilon\right)\rightarrow |f|^{p-2}f\mbox{ strongly in }L^{p'}\left((\xi_{i-1},\xi_i)\times Y^*_i\right).
		\end{equation*}
	\end{remark}
	\begin{proposition}\label{convergenceplaplacetype}
		Let $f\in L^p\left((0,1);L^p_\#(Y^*)\right)$ and extend it periodically in $y_1$-direction defining 
		\begin{equation*} \label{eq448}
		f^\varepsilon(x,y):=f\left(x,\frac{x}{\varepsilon^\alpha},\frac{y}{\varepsilon}\right)\in L^p(R^\varepsilon).
		\end{equation*}
		Then, 
		\begin{equation*}
		\mathcal{T}_\varepsilon^i f^\varepsilon\rightarrow f \mbox{ strongly in }L^p\left((\xi_{i-1},\xi_i)\times Y^*\right).
		\end{equation*}
	\end{proposition}
	\begin{proof}
		See \cite{nakasato1}.
	\end{proof}
	\begin{theorem}\label{thmressonant} \label{thmalpha<1}
		Let $(\varphi_\varepsilon)\subset W^{1,p}(R^\varepsilon_i)$, $1<p<\infty$, with $\left|\left|\left|\varphi_\varepsilon\right|\right|\right|_{W^{1,p}(R^\varepsilon_i)}$ uniformly bounded. 
		Then, there exists $\varphi^i\in W^{1,p}(\xi_{i-1},\xi_i)$ and $\varphi_1^i\in L^p((\xi_{i-1},\xi_i);W^{1,p}_\#(Y^*_i))$ such that (up to a subsequence) 
		\begin{eqnarray*}
			&&\mathcal{T}_\varepsilon^i \varphi_\varepsilon \to \varphi^i\mbox{ strongly in } L^p\left((\xi_{i-1},\xi_i);W^{1,p}(Y^*_i)\right),\\
			&&\mathcal{T}_\varepsilon^i\partial_x\varphi_\varepsilon\rightharpoonup \partial_x \varphi^i+\partial_{y_1}\varphi_1^i\mbox{ weakly in }L^p\left((\xi_{i-1},\xi_i)\times Y^*_i\right),\\
			&&\mathcal{T}_\varepsilon^i\partial_y\varphi_\varepsilon\rightharpoonup \partial_{y_2}\varphi_1^i\mbox{ weakly in }L^p\left((\xi_{i-1},\xi_i)\times Y^*_i\right).
		\end{eqnarray*}
	\end{theorem}
	\begin{proof}
		See \cite[Theorem 3.1 and 4.1]{AM2} respectively.  
	\end{proof}
	\subsubsection{Locally Periodic Unfolding}\label{subsectionlocallyperiodic}
	
	Now, let us set the locally periodic unfolding operator seeing some properties that will be needed in the sequel. 
	\begin{definition}\label{LPunfoldingdefinition}
		We define the locally periodic unfolding operator $T_\varepsilon^{lp}$ acting on $\varphi$, as the function  $T_\varepsilon^{lp}\varphi$ defined in $(0,1)\times (0,L)\times (0,G_1)$ by expression 
		\begin{equation*}
		T_\varepsilon^{lp}\varphi(x,y_1,y_2)=
		\widetilde{\varphi}\left(\varepsilon\left[\dfrac{x}{\varepsilon}\right]L+\varepsilon y_1,\varepsilon y_2\right)\mbox{ for }(x,y_1,y_2)\in (0,1)\times (0,L)\times (0,G_1),
		\end{equation*}
		where $\widetilde{\cdot}$ denotes the extesion by zero to the whole space.
	\end{definition}
	As in classical periodic homogenization, we have the unfolding operator reflecting two scales. The macroscopic one, denoted by $x$ which gives the position in the interval $(0,1)$, and the microscopic scale given by $(y_1,y_2)$ which sets the position in the cell $(0,L)\times (0,G_1)$. However, due to the locally periodic
	oscillations of the domain $R^\varepsilon$, the definition given here differs from the usual ones. In this case, we do not have a fixed cell that describes the domain $R^\varepsilon$ which makes the extesion by zero needed.	

	\begin{theorem}\label{locallyperiodicconvergence}
		Let $\varphi_\varepsilon\in W^{1,p}(R^\varepsilon)$ for $1<p<\infty$ such that $|||\varphi_\varepsilon|||_{W^{1,p}(R^\varepsilon)}$ is uniformly bounded. Then, there exists $\varphi\in W^{1,p}(0,1)$ such that, up to subsequences,
		$$
		T_\varepsilon^{lp}\varphi_\varepsilon\rightharpoonup \varphi\chi_{(0,1)\times Y^*(x)},
		$$
		weakly in $L^p\left((0,1)\times (0,L)\times (0,G_1)\right)$, 
		where $\chi_{(0,1)\times Y^*(x)}$ is the characteristic function of $(0,1)\times Y^*(x)$.
	\end{theorem}
	\begin{proof}
		See \cite{Villanueva2016}, Theorem 2.3.9.
	\end{proof}
	\begin{remark}
		We point out that the convergence above can not be improved because of the definition of locally periodic unfolding operator.
	\end{remark}

	\begin{proposition}\label{propostionlpconvbasic}
		\begin{enumerate}
			\item Let $\varphi\in L^1(R^\varepsilon)$. Then,
			$$
			\dfrac{1}{L}\int_{(0,1)\times (0,L)\times (0,G_1)}T_\varepsilon^{lp}\varphi(x,y_1,y_2) dx dy_1dy_2=\dfrac{1}{\varepsilon}\int_{R^\varepsilon}\varphi(x,y) dxdy.
			$$
			\item Let $\varphi\in L^p(0,1)$. Then,
			$$
			T_\varepsilon^{lp}\varphi\to \chi_{(0,1)\times Y^*(x)}\varphi \mbox{ strongly in }L^p\left((0,1)\times (0,L)\times (0,G_1)\right).
			$$
		\end{enumerate}
	\end{proposition}
	\begin{proof}
		See \cite[Propositions 2.2.5 and 2.3.6]{Villanueva2016}.
	\end{proof}
	\begin{proposition}\label{propLPunfolfingconv}
		Let $\varphi_\varepsilon\in L^p(R^\varepsilon)$ such that 
		$$
		T_\varepsilon^{lp}\varphi_\varepsilon\rightharpoonup \chi_{(0,1)\times Y^*(x)}\varphi \mbox{ weakly in }L^p\left((0,1)\times (0,L)\times (0,G_1)\right),
		$$
		where $\varphi(x,y_1,y_2)=\varphi(x)$. Then, 
		$$
		\dfrac{L}{\varepsilon}\int_0^{\varepsilon G_\varepsilon(\cdot)}\varphi_\varepsilon(\cdot,y) dy\rightharpoonup |Y^*(\cdot)|\varphi\mbox{ weakly in }L^p(0,1).
		$$ 
	\end{proposition}
	\begin{proof}
		Notice that
		$$
		\dfrac{1}{L}\int_{(0,1)\times (0,L)\times (0,G_1)}T_\varepsilon^{lp}\varphi_\varepsilon T_\varepsilon^{lp}\psi(x) dx dy_1dy_2\to \dfrac{1}{L}\int_{(0,1)\times (0,L)\times (0,G_1)}\varphi(x)\psi(x)\chi_{(0,1)\times Y^*(x)} dx dy_1dy_2,
		$$
		for all $\psi\in L^{p'}(0,1).$
		By the Proposition \ref{propostionlpconvbasic}, we have
		$$
		\begin{gathered}
		\dfrac{1}{L}\int_{(0,1)\times (0,L)\times (0,G_1)}T_\varepsilon^{lp}\varphi_\varepsilon T_\varepsilon^{lp}\psi(x) dx dy_1dy_2=\dfrac{1}{\varepsilon}\int_{R^\varepsilon}\varphi_\varepsilon(x,y)\psi(x) dxdy\\=\int_0^1\left(\dfrac{1}{\varepsilon}\int_0^{\varepsilon G_\varepsilon(x)}\varphi_\varepsilon(x,y)dy\right)\psi(x)dx
		\end{gathered}
		$$
		and
		$$
		\dfrac{1}{L}\int_{(0,1)\times (0,L)\times (0,G_1)}\varphi(x)\psi(x)\chi_{(0,1)\times Y^*(x)} dx dy_1dy_2=\dfrac{1}{L}\int_0^1 |Y^*(x)|\varphi(x)\psi(x) dx,
		$$
		for all $\psi\in L^{p'}(0,1)$.
		Thus,
		$$
		\dfrac{1}{\varepsilon}\int_0^{\varepsilon G_\varepsilon(x)}\varphi_\varepsilon(x,y)dy\rightharpoonup \dfrac{1}{L} |Y^*(x)|\varphi(x)
		$$
		weakly in $L^p(0,1)$.
	\end{proof}

	\section{A domain dependence result}\label{secdomdep}
	
	In the next we analyze how the solutions of \eqref{problem} depends on the function $G_\varepsilon$.  
	Let us take 
	$$
	G_\varepsilon(x)=G\left(x,\dfrac{x}{\varepsilon}\right) \quad \mbox{ and } \quad \hat{G}_\varepsilon(x)=\hat{G}\left(x,\dfrac{x}{\varepsilon}\right)
	$$ 
	satisfying hypothesis \hyperref[hypG]{(H)} and considering the associated thin domains $R^\varepsilon$ and $\hat{R}^\varepsilon$ by
	\begin{eqnarray*}
		R^\varepsilon=\left\lbrace (x,y)\in\mathbb{R}^2:x\in(0,1),0<y<\varepsilon G_\varepsilon(x) \right\rbrace \quad \textrm{ and } \\
		\hat{R}^\varepsilon=\left\lbrace (x,y)\in\mathbb{R}^2:x\in(0,1),0<y<\varepsilon \hat{G}_\varepsilon(x) \right\rbrace.
	\end{eqnarray*}
	
	Now, let $u_\varepsilon$ and $\hat{u}_\varepsilon$ be the solutions of \eqref{problem} for the domains $R^\varepsilon$ and $\hat{R}^\varepsilon$ respectively with $f^\varepsilon\in L^{p'}(\mathbb{R}^2)$. We have the following result.
	
	\begin{theorem}\label{thmapprox}
		Let $G_\varepsilon$ and $\hat{G}_\varepsilon$ be piecewise $C^1$ functions satisfying \hyperref[hypG]{(H)} with $$\|G_\varepsilon-\hat G_\varepsilon\|_{L^\infty(0,1)}\leq \delta.$$ 
		Assume also 
		$f^\varepsilon\in L^{p'}(\mathbb{R}^2)$ satisfying $\|f^\varepsilon\|_{L^p(\mathbb{R}^2)}\leq 1$.
		
		Then, there exists a positive real function $\rho:[0,\infty)\mapsto [0,\infty)$ such that 
		\begin{equation}\label{thmapproxineq}
		|||u_\varepsilon-\hat{u}_\varepsilon|||^p_{W^{1,p}(R^\varepsilon\cap \hat{R}^\varepsilon)}+|||u_\varepsilon|||^p_{W^{1,p}(R^\varepsilon\backslash \hat{R}^\varepsilon)}+|||\hat{u}_\varepsilon|||^p_{W^{1,p}(\hat{R}^\varepsilon\backslash R^\varepsilon)}\leq \rho(\delta),
		\end{equation}
		with $\rho(\delta)\to 0$ as $\delta\to0$ uniformly for all $\varepsilon>0$.
	\end{theorem}
	\begin{remark}
		The important part of this result is that the function $\rho(\delta)$ does not depend on $\varepsilon$. As we will see, it only depends on the positive constants $G_0$ and $G_1$.
	\end{remark}
	
	In order to prove Theorem \ref{thmapproxineq}, we use the fact that $u_\varepsilon$ and $\hat{u}_\varepsilon$ are minimizer of the the functionals
	\begin{equation}\label{functionals}
	\begin{gathered}
	V_\varepsilon(\varphi)=\dfrac{1}{p\,\varepsilon}\int_{R^\varepsilon}\left(|\nabla \varphi|^p+|\varphi|^p\right)dxdy-\dfrac{1}{\varepsilon}\int_{R^\varepsilon}f^\varepsilon\varphi dxdy\\
	\hat{V}_\varepsilon(\hat{\varphi})=\dfrac{1}{p\,\varepsilon}\int_{\hat{R}^\varepsilon}\left(|\nabla \hat\varphi|^p+|\hat\varphi|^p\right)dxdy-\dfrac{1}{\varepsilon}\int_{\hat{R}^\varepsilon}f^\varepsilon\hat\varphi dxdy
	\end{gathered}
	\end{equation}
	that is, 
	\begin{equation*}
	\begin{gathered}
	V_\varepsilon(u_\varepsilon)=\min_{\varphi\in W^{1,p}(R^\varepsilon)}V_\varepsilon(\varphi)
	\quad \textrm{ and } \quad 
	\hat{V}_\varepsilon(\hat{u}_\varepsilon)=\min_{\hat\varphi\in W^{1,p}(\hat R^\varepsilon)}\hat V_\varepsilon(\hat\varphi).
	\end{gathered}
	\end{equation*}
	
	We will need to evaluate the minimizers plugging them into different functionals.  For this, we set the following operators introduced in \cite{AP10}:
	\begin{equation}\label{operatorcomparison}
	\begin{gathered}
	P_{1+\eta}:W^{1,p}(U)\mapsto W^{1,p}\left(U(1+\eta)\right)\\
	\left(P_{1+\eta}\varphi\right)(x,y)=\varphi\left(x,\dfrac{y}{1+\eta}\right), \quad (x,y)\in U(1+\eta),
	\end{gathered}
	\end{equation}
	where 
	\begin{equation}\label{rescaleddomain}
	U(1+\eta)=\left\lbrace \left(x,(1+\eta)y\right)\in \mathbb{R}^2:(x,y)\in U\right\rbrace
	\end{equation}
	and $U\subset \mathbb{R}^2$ is an arbitrary open set. 
	We also consider the following norm in $W^{1,p}(U)$
	\begin{equation}\label{etanorm}
	||w||^p_{W^{1,p}_{1+\eta}(U)}=\dfrac{1}{1+\eta}\left[\left|\left|w\right|\right|^p_{L^p(U)}+\left|\left|K_{1+\eta}\nabla w\right|\right|^p_{L^p(U)}\right]
	\end{equation}
	where 
	$$
	K_{1+\eta}=\left(\begin{array}{cc}
	1 & 0\\
	0 & 1+\eta
	\end{array}\right).
	$$
	
	We can easily see that
	\begin{equation}\label{relationnorm}
	||w||^p_{W^{1,p} (U)}=||P_{1+\eta}w||^p_{W^{1,p}_{1+\eta}(U(1+\eta))}
	\end{equation}
	and
	\begin{equation*}\label{equivalencenorm}
	\dfrac{1}{1+\eta}||w||^p_{W^{1,p}(U)}\leq ||w ||_{W^{1,p}_{1+\eta}(U)}\leq (1+\eta) ||w||^p_{W^{1,p}(U)} \quad \textrm{ as } \eta \geq 0.
	\end{equation*}
	Also, we need the following result about the behavior of the solutions near of the oscillating boundary.
	\begin{lemma}\label{lemmaestimates}
		Let $u_\varepsilon$ be the solution of problem \eqref{problem} and let $P_{1+\eta}$ be the operator given by \eqref{operatorcomparison}. Then, there exists a positive function $\rho=\rho(p,\eta)$ satisfying $\rho(p,\eta)\to 0$ as $\eta\to 0$, such that 
		\begin{equation*} 
		\begin{gathered}
		|||u_\varepsilon |||_{W^{1,p}\left(R^\varepsilon\backslash R^\varepsilon\left(\frac{1}{1+\eta}\right)\right)}^p+|||u_\varepsilon |||_{W^{1,p}\left(R^\varepsilon(1+\eta)\backslash R^\varepsilon \right)}^p
		+|||P_{1+\eta}u_\varepsilon-u_\varepsilon |||_{W^{1,p}\left(R^\varepsilon\right)}^p
		\leq \rho(p,\eta),
		\end{gathered}
		\end{equation*}
		for $1<p<\infty$.
	\end{lemma}
	\begin{proof}
		Since $\eta>0$, we have that $R^\varepsilon\left(\frac{1}{1+\eta}\right)\subset R^\varepsilon$. Then,
		\begin{equation}\label{lemmacompineq01}
		\begin{gathered}
		V(u_\varepsilon)=\dfrac{1}{p}||| u_\varepsilon|||_{W^{1,p}(R^\varepsilon)}^p-\dfrac{1}{\varepsilon}\int_{R^\varepsilon}f^\varepsilon u_\varepsilon dxdy \qquad\qquad\qquad\qquad\qquad\qquad\\
		=\dfrac{1}{p}|||u_\varepsilon |||_{W^{1,p}\left(R^\varepsilon\backslash R^\varepsilon\left(\frac{1}{1+\eta}\right)\right)}^p+\dfrac{1}{p}|||u_\varepsilon |||_{W^{1,p}\left(R^\varepsilon\left(\frac{1}{1+\eta}\right)\right)}^p-\dfrac{1}{\varepsilon}\int_{R^\varepsilon}f^\varepsilon u_\varepsilon dxdy\\
		=\dfrac{1}{p}|||u_\varepsilon |||_{W^{1,p}\left(R^\varepsilon\backslash R^\varepsilon\left(\frac{1}{1+\eta}\right)\right)}^p+\dfrac{1}{p}|||P_{1+\eta}u_\varepsilon |||_{W^{1,p}_{1+\eta}\left(R^\varepsilon\right)}^p-\dfrac{1}{\varepsilon}\int_{R^\varepsilon}f^\varepsilon u_\varepsilon dxdy\\
		\geq \dfrac{1}{p}|||u_\varepsilon |||_{W^{1,p}\left(R^\varepsilon\backslash R^\varepsilon\left(\frac{1}{1+\eta}\right)\right)}^p+\dfrac{1}{p(1+\eta)}|||P_{1+\eta}u_\varepsilon |||_{W^{1,p}\left(R^\varepsilon\right)}^p-\dfrac{1}{\varepsilon}\int_{R^\varepsilon}f^\varepsilon u_\varepsilon dxdy.
		\end{gathered}
		\end{equation}
		
		Now, let us first assume $p\geq 2$. We use the notations of Corollary \ref{corolarioplaplaciano} to simplify proofs.
		By Proposition \ref{proposicaoplaplacian01}, \eqref{functionals} and \eqref{variationalproblem} for $\varphi=P_{1+\eta}u_\varepsilon-u_\varepsilon$, we get
		\begin{equation}\label{lemmacompineq02}
		\begin{gathered}
		|||P_{1+\eta}u_\varepsilon |||_{W^{1,p}\left(R^\varepsilon\right)}^p\geq ||| u_\varepsilon|||_{W^{1,p}(R^\varepsilon)}^p
		+\dfrac{p}{\varepsilon}\int_{R^\varepsilon}\left[a_p(\nabla u_\varepsilon)\nabla \left(P_{1+\eta}u_\varepsilon-u_\varepsilon\right)\right.\\
		\qquad\qquad\qquad\qquad\left.+a_p(u_\varepsilon)\left(P_{1+\eta}u_\varepsilon-u_\varepsilon\right)\right]dxdy+c_p|||P_{1+\eta}u_\varepsilon-u_\varepsilon |||_{W^{1,p}\left(R^\varepsilon\right)}^p\\
		=pV(u_\varepsilon)+\dfrac{p}{\varepsilon}\int_{R^\varepsilon}f^\varepsilon u_\varepsilon dxdy
		+\dfrac{p}{\varepsilon}\int_{R^\varepsilon}f^\varepsilon\left(P_{1+\eta}u_\varepsilon-u_\varepsilon\right)dxdy+c_p|||P_{1+\eta}u_\varepsilon-u_\varepsilon |||_{W^{1,p}\left(R^\varepsilon\right)}^p\\
		=pV(u_\varepsilon)
		+\dfrac{p}{\varepsilon}\int_{R^\varepsilon}f^\varepsilon P_{1+\eta}u_\varepsilon dxdy+c_p|||P_{1+\eta}u_\varepsilon-u_\varepsilon |||_{W^{1,p}\left(R^\varepsilon\right)}^p.
		\end{gathered}
		\end{equation}
		Putting together \eqref{lemmacompineq01} and \eqref{lemmacompineq02}, we obtain
		\begin{equation*}
		\begin{gathered}
		V(u_\varepsilon)\geq \dfrac{1}{p}|||u_\varepsilon |||_{W^{1,p}\left(R^\varepsilon\backslash R^\varepsilon\left(\frac{1}{1+\eta}\right)\right)}^p+\dfrac{1}{p(1+\eta)}|||P_{1+\eta}u_\varepsilon |||_{W^{1,p}\left(R^\varepsilon\right)}^p-\dfrac{1}{\varepsilon}\int_{R^\varepsilon}f^\varepsilon u_\varepsilon dxdy\\
		\geq \dfrac{1}{p}|||u_\varepsilon |||_{W^{1,p}\left(R^\varepsilon\backslash R^\varepsilon\left(\frac{1}{1+\eta}\right)\right)}^p+\dfrac{1}{1+\eta}V(u_\varepsilon)\\
		+\dfrac{1}{\varepsilon(1+\eta)}\int_{R^\varepsilon}f^\varepsilon P_{1+\eta}u_\varepsilon dxdy+\dfrac{c_p}{1+\eta}|||P_{1+\eta}u_\varepsilon-u_\varepsilon |||_{W^{1,p}\left(R^\varepsilon\right)}^p-\dfrac{1}{\varepsilon}\int_{R^\varepsilon}f^\varepsilon u_\varepsilon dxdy.
		\end{gathered}
		\end{equation*}
		Consequently
		\begin{equation*}
		\begin{gathered}
		\dfrac{\eta}{1+\eta}V(u_\varepsilon)\geq \dfrac{1}{p}|||u_\varepsilon |||_{W^{1,p}\left(R^\varepsilon\backslash R^\varepsilon\left(\frac{1}{1+\eta}\right)\right)}^p
		\\+\dfrac{1}{\varepsilon}\int_{R^\varepsilon}f^\varepsilon \left[\dfrac{P_{1+\eta}u_\varepsilon}{(1+\eta)}-u_\varepsilon\right] dxdy+\dfrac{c_p}{1+\eta}|||P_{1+\eta}u_\varepsilon-u_\varepsilon |||_{W^{1,p}\left(R^\varepsilon\right)}^p 
		\end{gathered}
		\end{equation*}
		which implies
		\begin{equation}\label{lemmacompineq03} 
		\begin{gathered}
		\dfrac{1}{p}|||u_\varepsilon |||_{W^{1,p}\left(R^\varepsilon\backslash R^\varepsilon\left(\frac{1}{1+\eta}\right)\right)}^p
		+\dfrac{c_p}{1+\eta}|||P_{1+\eta}u_\varepsilon-u_\varepsilon |||_{W^{1,p}\left(R^\varepsilon\right)}^p\\\leq \dfrac{\eta}{1+\eta}V(u_\varepsilon) +\dfrac{1}{\varepsilon}\int_{R^\varepsilon}f^\varepsilon \left[u_\varepsilon-\dfrac{P_{1+\eta}u_\varepsilon}{(1+\eta)}\right] dxdy.
		\end{gathered}
		\end{equation}
		
		Now, let us analyze the integral:
		$$
		\dfrac{1}{\varepsilon}\int_{R^\varepsilon}f^\varepsilon \left[u_\varepsilon-\dfrac{P_{1+\eta}u_\varepsilon}{(1+\eta)}\right] dxdy.
		$$
		To do this, notice that 
		$$
		u_\varepsilon(x,y)-(P_{1+\eta}u_\varepsilon)(x,y)=u_\varepsilon(x,y)-u_\varepsilon\left(x,\dfrac{y}{1+\eta}\right)=\int_{\frac{y}{1+\eta}}^y\partial_y u_\varepsilon(x,s)ds,
		$$
		which implies
		$$
		\left|u_\varepsilon(x,y)-(P_{1+\eta}u_\varepsilon)(x,y)\right|\leq \left[\int_{\frac{y}{1+\eta}}^y|\partial_y u_\varepsilon(x,s)|^pds\right]^{1/p}\left(\dfrac{\eta y}{(1+\eta)}\right)^{1/p'}
		$$
		putting the power $p$, multiplying by $1/\varepsilon$, integrating between $0$ and $\varepsilon G_\varepsilon(x)$ and using that $(y/(1+\eta),y)\subset (\varepsilon G_\varepsilon(x))$, we get
		\begin{equation*}
		\begin{gathered}
		\dfrac{1}{\varepsilon}\int_0^{\varepsilon G_\varepsilon(x)}\left|u_\varepsilon(x,y)-(P_{1+\eta}u_\varepsilon)(x,y)\right|^p dy\leq \left[\dfrac{1}{\varepsilon}\int_{0}^{\varepsilon G_\varepsilon(x)}|\partial_y u_\varepsilon(x,s)|^pds\right]\left(\dfrac{\eta }{1+\eta}\right)^{p-1}\dfrac{(\varepsilon G_\varepsilon(x))^{p} }{p}.
		\end{gathered}
		\end{equation*}
		
		Thus, we have
		$$
		|||u_\varepsilon-P_{1+\eta}u_\varepsilon|||_{L^p(R^\varepsilon)}\leq |||\partial_yu_\varepsilon|||_{L^p(R^\varepsilon)}\left(\dfrac{\eta }{1+\eta}\right)^{1/p'}\dfrac{G_1}{p^{1/p}},
		$$
		for $\varepsilon<1$.
		Consequently, we get
		\begin{equation}\label{lemmacompineq04}
		\begin{gathered}
		\left|\dfrac{1}{\varepsilon}\int_{R^\varepsilon}f^\varepsilon \left[u_\varepsilon-\dfrac{P_{1+\eta}u_\varepsilon}{(1+\eta)}\right] dxdy\right|\\
		\leq \dfrac{\eta}{\varepsilon(1+\eta)}\int_{R^\varepsilon}|f^\varepsilon u_\varepsilon| dxdy+\dfrac{1}{\varepsilon(1+\eta)}\int_{R^\varepsilon}\left|f^\varepsilon u_\varepsilon-f^\varepsilon P_{1+\eta}u_\varepsilon \right| dxdy\\
		\leq \dfrac{\eta}{1+\eta}|||f^\varepsilon|||_{L^{p'}(R^\varepsilon)}|||u_\varepsilon|||_{L^{p}(R^\varepsilon)}+|||f^\varepsilon|||_{L^{p'}(R^\varepsilon)}|||\partial_yu_\varepsilon|||_{L^p(R^\varepsilon)} \dfrac{\eta^{1/p'} }{(1+\eta)^{1+1/p'}}\dfrac{G_1}{p^{1/p}}.
		\end{gathered}
		\end{equation}
		Hence, due Proposition \ref{uniformlimtation}, \eqref{lemmacompineq03} and \eqref{lemmacompineq04}, one gets
		\begin{equation}\label{lemmacompp>201}
		\begin{gathered}
		\dfrac{1}{p}|||u_\varepsilon |||_{W^{1,p}\left(R^\varepsilon\backslash R^\varepsilon\left(\frac{1}{1+\eta}\right)\right)}^p
		+c_p|||P_{1+\eta}u_\varepsilon-u_\varepsilon |||_{W^{1,p}\left(R^\varepsilon\right)}^p\\
		\leq\dfrac{\eta}{1+\eta}c+\dfrac{\eta}{1+\eta}c+\dfrac{\eta^{1/p'} }{(1+\eta)^{1+1/p'}}c\\
		\leq c\eta+c\eta^{1/p'}. 
		\end{gathered}
		\end{equation}
		
		On the other hand, we have
		\begin{equation*}\label{lemmacompineq05}
		\begin{gathered}
		V(u_\varepsilon)=\dfrac{1}{p}||| u_\varepsilon|||_{W^{1,p}(R^\varepsilon)}^p-\dfrac{1}{\varepsilon}\int_{R^\varepsilon}f^\varepsilon u_\varepsilon dxdy \\
		=\dfrac{1}{p}||| P_{1+\eta}u_\varepsilon|||_{W^{1,p}_{1+\eta}(R^\varepsilon(1+\eta))}^p-\dfrac{1}{\varepsilon}\int_{R^\varepsilon}f^\varepsilon u_\varepsilon dxdy \\
		=\dfrac{1}{p}|||P_{1+\eta}u_\varepsilon |||_{W^{1,p}_{1+\eta}\left(R^\varepsilon(1+\eta)\backslash R^\varepsilon\right)}^p+\dfrac{1}{p}|||P_{1+\eta}u_\varepsilon |||_{W^{1,p}_{1+\eta}\left(R^\varepsilon\right)}^p-\dfrac{1}{\varepsilon}\int_{R^\varepsilon}f^\varepsilon u_\varepsilon dxdy\\
		\geq \dfrac{1}{p(1+\eta)}\left[|||P_{1+\eta}u_\varepsilon |||_{W^{1,p}\left(R^\varepsilon(1+\eta)\backslash R^\varepsilon)\right)}^p+|||P_{1+\eta}u_\varepsilon |||_{W^{1,p}\left(R^\varepsilon\right)}^p\right]-\dfrac{1}{\varepsilon}\int_{R^\varepsilon}f^\varepsilon u_\varepsilon dxdy.
		\end{gathered}
		\end{equation*}
		Hence, due to \eqref{lemmacompineq02}, we get 
		\begin{equation*}
		\begin{gathered}
		V(u_\varepsilon)\geq \dfrac{1}{p(1+\eta)}\left[|||P_{1+\eta}u_\varepsilon |||_{W^{1,p}\left(R^\varepsilon(1+\eta)\backslash R^\varepsilon)\right)}^p+|||P_{1+\eta}u_\varepsilon |||_{W^{1,p}\left(R^\varepsilon\right)}^p\right]-\dfrac{1}{\varepsilon}\int_{R^\varepsilon}f^\varepsilon u_\varepsilon dxdy\\
		\geq\dfrac{1}{p(1+\eta)}|||P_{1+\eta}u_\varepsilon |||_{W^{1,p}\left(R^\varepsilon(1+\eta)\backslash R^\varepsilon)\right)}^p+\dfrac{1}{(1+\eta)}V(u_\varepsilon)
		+\dfrac{1}{(1+\eta)\varepsilon}\int_{R^\varepsilon}f^\varepsilon P_{1+\eta}u_\varepsilon dxdy\\+c_p|||P_{1+\eta}u_\varepsilon-u_\varepsilon |||_{W^{1,p}\left(R^\varepsilon\right)}^p-\dfrac{1}{\varepsilon}\int_{R^\varepsilon}f^\varepsilon u_\varepsilon dxdy,
		\end{gathered}
		\end{equation*}
		and then,
		\begin{equation*}
		\begin{gathered}
		\dfrac{1}{p(1+\eta)}|||P_{1+\eta}u_\varepsilon |||_{W^{1,p}\left(R^\varepsilon(1+\eta)\backslash R^\varepsilon)\right)}^p+c_p|||P_{1+\eta}u_\varepsilon-u_\varepsilon |||_{W^{1,p}\left(R^\varepsilon\right)}^p\\\leq \dfrac{\eta}{1+\eta}V(u_\varepsilon)+\dfrac{1}{\varepsilon}\int_{R^\varepsilon}f^\varepsilon \left(u_\varepsilon -\dfrac{P_{1+\eta}u_\varepsilon}{(1+\eta)} \right) dxdy.
		\end{gathered}
		\end{equation*}
		Thus, due Proposition \ref{uniformlimtation} and \eqref{lemmacompineq04}, we get for $p>2$ that 
		\begin{equation}\label{lemmacompp>202}
		\begin{gathered}
		\dfrac{1}{p}|||P_{1+\eta}u_\varepsilon |||_{W^{1,p}\left(R^\varepsilon(1+\eta)\backslash R^\varepsilon \right)}^p
		+c_p|||P_{1+\eta}u_\varepsilon-u_\varepsilon |||_{W^{1,p}\left(R^\varepsilon\right)}^p\\
		\leq c\eta+c\eta^{1/p'}.
		\end{gathered}
		\end{equation}

		Notice that to the case $p>2$, we have mainly estimated the term
		$
		|x-y|^p.
		$
		Now, for the case $1<p<2$, we have to estimate  
		$
		(1+|x|+|y|)^{p-2}|x-y|^2
		$
		in view of Propositions \ref{proposicaoplaplaciano} and \ref{proposicaoplaplacian01}.
		Indeed, we can argue as in \eqref{lemmacompp>201} and \eqref{lemmacompp>202}, to get, for $1<p<2$ that
		\begin{equation}\label{lemmacompp<201}
		\begin{gathered}
		\dfrac{1}{p}|||u_\varepsilon |||_{W^{1,p}\left(R^\varepsilon\backslash R^\varepsilon\left(\frac{1}{1+\eta}\right)\right)}^p
		+\dfrac{c_p}{\varepsilon}\int_{R^\varepsilon}\left|\nabla P_{1+\eta}u_\varepsilon-\nabla  u_\varepsilon\right|^2\left(1+|\nabla P_{1+\eta}u_\varepsilon|+|\nabla u_\varepsilon|\right)^{p-2} dxdy\\
		+\dfrac{c_p}{\varepsilon}\int_{R^\varepsilon}\left|P_{1+\eta}u_\varepsilon-u_\varepsilon\right|^2\left(1+|P_{1+\eta}u_\varepsilon|+|u_\varepsilon|\right)^{p-2} dxdy\leq c\eta+c\eta^{1/p'}
		\end{gathered}
		\end{equation}
		and 
		\begin{equation*}\label{lemmacompp<202}
		\begin{gathered}
		\dfrac{1}{p}|||P_{1+\eta}u_\varepsilon |||_{W^{1,p}\left(R^\varepsilon(1+\eta)\backslash R^\varepsilon \right)}^p
		+\dfrac{c_p}{\varepsilon}\int_{R^\varepsilon}\left|\nabla P_{1+\eta}u_\varepsilon-\nabla  u_\varepsilon\right|^2\left(1+|\nabla P_{1+\eta}u_\varepsilon|+|\nabla u_\varepsilon|\right)^{p-2} dxdy\\
		+\dfrac{c_p}{\varepsilon}\int_{R^\varepsilon}\left|P_{1+\eta}u_\varepsilon-u_\varepsilon\right|^2\left(1+|P_{1+\eta}u_\varepsilon|+|u_\varepsilon|\right)^{p-2} dxdy\leq c \eta+c\eta^{p-1}. 
		\end{gathered}
		\end{equation*}
		
		Now, notice that 
		\begin{equation*}
		\begin{gathered}
		|||P_{1+\eta}u_\varepsilon -u_\varepsilon|||^p_{W^{1,p}\left(R^\varepsilon\right)} d\leq \left(\dfrac{1}{\varepsilon}\int_{R^\varepsilon}\left|\nabla P_{1+\eta}u_\varepsilon-\nabla  u_\varepsilon\right|^2\left(1+|\nabla P_{1+\eta}u_\varepsilon|+|\nabla u_\varepsilon|\right)^{p-2} dxdy\right)^{p/2}\\
		\quad\qquad\quad\qquad\cdot\left[\dfrac{1}{\varepsilon}\int_{R^\varepsilon} \left(1+|\nabla P_{1+\eta}u_\varepsilon|+|\nabla u_\varepsilon|\right)^{p} dxdy\right]^{(2-p)/2}\\
		+\left(\dfrac{1}{\varepsilon}\int_{R^\varepsilon}\left|P_{1+\eta}u_\varepsilon-  u_\varepsilon\right|^2\left(1+|  P_{1+\eta}u_\varepsilon|+|  u_\varepsilon|\right)^{p-2} dxdy\right)^{p/2}\\
		\quad\qquad\quad\qquad\cdot\left[\dfrac{1}{\varepsilon}\int_{R^\varepsilon} \left(1+|  P_{1+\eta}u_\varepsilon|+|  u_\varepsilon|\right)^{p} dxdy\right]^{(2-p)/2}.
		\end{gathered}
		\end{equation*} 
		
		Finally, putting together the last inequality and \eqref{lemmacompp<201}, we also obtain 		
		\begin{equation*}
		\dfrac{1}{p}|||u_\varepsilon |||_{W^{1,p}\left(R^\varepsilon\backslash R^\varepsilon\left(\frac{1}{1+\eta}\right)\right)}^p+|||P_{1+\eta}u_\varepsilon-u_\varepsilon |||_{W^{1,p}\left(R^\varepsilon\right)}^p\leq c\eta+c\eta^{1/p'}+\left[ c\eta+c\eta^{1/p'}\right]^{p/2}
		\end{equation*}
		for $1<p<2$ finishing the proof.
	\end{proof}

	Now, we are in condition to show Theorem \ref{thmapprox}. 
	\begin{proof}[Proof of Theorem \ref{thmapprox}]
	
		Taking $\eta=\delta/G_0$, we get under condition $\|G_\varepsilon-\hat G_\varepsilon\|\leq \delta$ that 
		\begin{equation}\label{subsets}
		\begin{gathered}
		R^\varepsilon\left(\frac{1}{1+\eta}\right)\subset \hat R^\varepsilon\subset R^\varepsilon(1+\eta) 
		\quad \textrm{ and } \quad 
		\hat R^\varepsilon\left(\frac{1}{1+\eta}\right)\subset  R^\varepsilon\subset \hat R^\varepsilon(1+\eta).
		\end{gathered}
		\end{equation}
		Applying Lemma \ref{lemmaestimates}, we get 
		\begin{equation}\label{eq1048}
		\begin{gathered}
		|||u_\varepsilon|||_{W^{1,p}(R^\varepsilon\backslash\hat R^\varepsilon)}^p\leq |||u_\varepsilon|||_{W^{1,p}\left(R^\varepsilon\backslash R^\varepsilon\left(\frac{1}{1+\eta}\right)\right)}^p\leq c\rho(\eta) \quad \textrm{ and } \\
		|||u_\varepsilon|||_{W^{1,p}(\hat R^\varepsilon\backslash  R^\varepsilon)}^p\leq |||u_\varepsilon|||_{W^{1,p}\left(\hat R^\varepsilon\backslash \hat R^\varepsilon\left(\frac{1}{1+\eta}\right)\right)}^p\leq c\rho(\eta).
		\end{gathered}
		\end{equation}
		Now, let us focus to the first term of \eqref{thmapproxineq}. We have
		\begin{equation}\label{eq1055}
		\begin{gathered}
		V_\varepsilon (u_\varepsilon)\leq V_\varepsilon\left(\left(P_{1+\eta} \hat u_\varepsilon\right)|_{R^\varepsilon}\right)\\
		=\dfrac{1}{p}|||\left(P_{1+\eta} \hat u_\varepsilon\right)|_{R^\varepsilon}|||_{W^{1,p}(R^\varepsilon)}-\dfrac{1}{\varepsilon}\int_{R^\varepsilon}f^\varepsilon \left(P_{1+\eta} \hat u_\varepsilon\right)|_{R^\varepsilon} dxdy\\
		\leq \dfrac{1}{p}|||P_{1+\eta} \hat u_\varepsilon |||_{W^{1,p}(\hat R^\varepsilon(1+\eta))}-\dfrac{1}{\varepsilon}\int_{\hat R^\varepsilon}f^\varepsilon  P_{1+\eta} \hat u_\varepsilon dxdy
		+\dfrac{1}{\varepsilon}\int_{\hat R^\varepsilon\backslash R^\varepsilon}f^\varepsilon  P_{1+\eta} \hat u_\varepsilon dxdy.
		\end{gathered}
		\end{equation}	
		But from the definition of $P_{1+\eta}$ (see \eqref{operatorcomparison}) and a change of variables, we get 
		\begin{equation}\label{1064lp}
		||| P_{1+\eta}\hat u_\varepsilon|||^p_{W^{1,p}\left(\hat R^\varepsilon(1+\eta)\right)}\leq (1+\eta)||| \hat u_\varepsilon|||^p_{W^{1,p}\left(\hat R^\varepsilon\right)}.
		\end{equation}
		From Lemma \ref{lemmaestimates} we get
		\begin{equation}\label{1072lp}
		\dfrac{1}{\varepsilon}\int_{\hat R^\varepsilon}f^\varepsilon \left(P_{1+\eta} \hat u_\varepsilon-\hat u_\varepsilon\right) dxdy\leq |||f^\varepsilon|||_{L^{p'}(\hat{R}^\varepsilon)}|||P_{1+\eta} \hat u_\varepsilon-\hat u_\varepsilon|||_{L^p(\hat R^\varepsilon)}\leq c\rho(\eta)^{1/p}.
		\end{equation}
		Also, by \eqref{subsets}, \eqref{eq1048} and Lemma \ref{lemmaestimates}, we obtain
		\begin{equation}\label{1076lp}
		\dfrac{1}{\varepsilon}\int_{\hat R^\varepsilon\backslash R\varepsilon}f^\varepsilon  P_{1+\eta}\hat u_\varepsilon dxdy\leq |||f^\varepsilon|||_{L^{p'}(\hat{R}^\varepsilon)}|||P_{1+\eta} \hat u_\varepsilon|||_{L^p(\hat R^\varepsilon\backslash R^\varepsilon)}\leq c\rho(\eta)^{1/p}.
		\end{equation}

		Hence, using \eqref{functionals}, \eqref{eq1055}, \eqref{1064lp}, Proposition \ref{uniformlimtation}, \eqref{1072lp}, \eqref{1076lp}, we get
		\begin{equation}\label{thm3.3ineq01}
		\begin{gathered}
		V_\varepsilon(u_\varepsilon)
		\leq \dfrac{(1+\eta)}{p}||| \hat u_\varepsilon|||^p_{W^{1,p}\left(\hat R^\varepsilon\right)}-\dfrac{1}{\varepsilon}\int_{\hat R^\varepsilon}f^\varepsilon  P_{1+\eta} \hat u_\varepsilon dxdy
		+\dfrac{1}{\varepsilon}\int_{\hat R^\varepsilon\backslash R^\varepsilon}f^\varepsilon  P_{1+\eta} \hat u_\varepsilon dxdy\\
		=(1+\eta)\hat{V}_\varepsilon(\hat{u}_\varepsilon)+\dfrac{(1+\eta)}{\varepsilon}\int_{\hat R^\varepsilon}f^\varepsilon \hat u_\varepsilon dxdy
		-\dfrac{1}{\varepsilon}\int_{\hat R^\varepsilon}f^\varepsilon P_{1+\eta} \hat u_\varepsilon dxdy
		+\dfrac{1}{\varepsilon}\int_{\hat R^\varepsilon\backslash R^\varepsilon}f^\varepsilon  P_{1+\eta} \hat u_\varepsilon dxdy\\
		=(1+\eta)\hat{V}_\varepsilon(\hat{u}_\varepsilon)+\dfrac{\eta}{\varepsilon}\int_{\hat R^\varepsilon}f^\varepsilon \hat u_\varepsilon dxdy
		+\dfrac{1}{\varepsilon}\int_{\hat R^\varepsilon}f^\varepsilon (\hat u_\varepsilon-P_{1+\eta} \hat u_\varepsilon) dxdy
		+\dfrac{1}{\varepsilon}\int_{\hat R^\varepsilon\backslash R^\varepsilon}f^\varepsilon  P_{1+\eta} \hat u_\varepsilon dxdy\\
		\leq(1+\eta)\hat{V}_\varepsilon(\hat{u}_\varepsilon)+\eta|||f^\varepsilon|||_{L^{p'}(\hat{R}^\varepsilon)} |||\hat u_\varepsilon|||_{L^{p}(\hat{R}^\varepsilon)}+c\rho(\eta)^{1/p}\\
		=(1+\eta)\hat{V}_\varepsilon(\hat{u}_\varepsilon)+ \bar{\rho}(\eta),
		\end{gathered}
		\end{equation}
		where $\bar{\rho}$ denotes a function such that $\bar\rho(\eta)\to0$ as $\eta\to0$.
		
		On the other hand, by \eqref{functionals}, \eqref{etanorm}, \eqref{relationnorm}, \eqref{subsets} and Proposition \ref{proposicaoplaplacian01}, we get, for $p\geq 2$, 
		\begin{equation}\label{thm3.3ineq02}
		\begin{gathered}
		V_\varepsilon(u_\varepsilon)=\dfrac{1}{p}|||u_\varepsilon |||_{W^{1,p}(R^\varepsilon)}^p-\dfrac{1}{\varepsilon}\int_{R^\varepsilon} f^\varepsilon u_\varepsilon dxdy\\
		= \dfrac{1}{p}|||P_{1+\eta}u_\varepsilon |||_{W^{1,p}_{1+\eta}(R^\varepsilon(1+\eta))}^p-\dfrac{1}{\varepsilon}\int_{R^\varepsilon} f^\varepsilon u_\varepsilon dxdy\\
		\geq \dfrac{1}{p(1+\eta)}|||P_{1+\eta}u_\varepsilon |||_{W^{1,p}(\hat R^\varepsilon)}^p-\dfrac{1}{\varepsilon}\int_{R^\varepsilon} f^\varepsilon u_\varepsilon dxdy\\
		\geq \dfrac{1}{p(1+\eta)}\left[|||\hat u_\varepsilon |||_{W^{1,p}(\hat R^\varepsilon)}^p+\dfrac{p}{\varepsilon}\int_{\hat{R}^\varepsilon}\left(a_p(\nabla \hat{u}_\varepsilon)\nabla(P_{1+\eta}u_\varepsilon-\hat u_\varepsilon)\right.\right.\\+\left.\left.a_p(\hat{u}_\varepsilon)(P_{1+\eta}u_\varepsilon-\hat u_\varepsilon)\right)dxdy+c_p|||P_{1+\eta}u_\varepsilon -\hat u_\varepsilon |||^p_{W^{1,p}(\hat R^\varepsilon)}\right] -\dfrac{1}{\varepsilon}\int_{R^\varepsilon} f^\varepsilon u_\varepsilon dxdy\\
		=\dfrac{1}{p(1+\eta)}\left[p\hat V(\hat u_\varepsilon)+\dfrac{p}{\varepsilon}\int_{\hat R^\varepsilon} f^\varepsilon \hat u_\varepsilon dxdy+\dfrac{p}{\varepsilon}\int_{\hat{R}^\varepsilon}f^\varepsilon(P_{1+\eta}u_\varepsilon-\hat u_\varepsilon)dxdy\right. \\+\left.c_p|||P_{1+\eta}u_\varepsilon -\hat u_\varepsilon |||^p_{W^{1,p}(\hat R^\varepsilon)}\right] -\dfrac{1}{\varepsilon}\int_{R^\varepsilon} f^\varepsilon u_\varepsilon dxdy\\
		=\dfrac{1}{(1+\eta)}\hat V(\hat u_\varepsilon)+\dfrac{1}{\varepsilon}\int_{\hat R^\varepsilon} f^\varepsilon\dfrac{1}{1+\eta}P_{1+\eta}u_\varepsilon dxdy-\dfrac{1}{\varepsilon}\int_{R^\varepsilon} f^\varepsilon u_\varepsilon dxdy+\dfrac{c_p}{p(1+\eta)}|||P_{1+\eta}u_\varepsilon -\hat u_\varepsilon |||^p_{W^{1,p}(\hat R^\varepsilon)}.
		\end{gathered}
		\end{equation}

		Now, due \eqref{lemmacompineq04}, a H\"older's inequality and Lemma \ref{lemmaestimates}, we obtain
		\begin{equation}\label{thm3.3ineq03}
		\begin{gathered}
		\left|\dfrac{1}{\varepsilon}\int_{\hat R^\varepsilon} f^\varepsilon\dfrac{1}{1+\eta}P_{1+\eta}u_\varepsilon dxdy-\dfrac{1}{\varepsilon}\int_{R^\varepsilon} f^\varepsilon u_\varepsilon dxdy\right|\\
		\leq \left|\dfrac{1}{\varepsilon}\int_{\hat R^\varepsilon\backslash R^\varepsilon} f^\varepsilon P_{1+\eta}u_\varepsilon dxdy\right|+\left|\dfrac{1}{\varepsilon}\int_{R^\varepsilon\backslash \hat R^\varepsilon} f^\varepsilon P_{1+\eta}u_\varepsilon dxdy\right|\\
		\qquad\qquad\qquad\qquad\qquad+\left|\dfrac{1}{(1+\eta)\varepsilon} \int_{  R^\varepsilon} f^\varepsilon P_{1+\eta}u_\varepsilon dxdy-\dfrac{1}{\varepsilon}\int_{R^\varepsilon} f^\varepsilon u_\varepsilon dxdy\right|
		\leq c\rho(\delta)^{1/p}.
		\end{gathered}
		\end{equation}
		First, one can put together \eqref{thm3.3ineq01} and \eqref{thm3.3ineq02}, and then use \eqref{thm3.3ineq03} to lead us to 
		\begin{equation*}
		\dfrac{c_p}{p(1+\eta)}|||P_{1+\eta}u_\varepsilon -\hat u_\varepsilon |||^p_{W^{1,p}(\hat R^\varepsilon)}\leq \dfrac{\eta^{2}+2\eta}{1+\eta}\hat V_\varepsilon(\hat u_\varepsilon)+\rho(\delta)^{1/p}+\bar\rho(\delta),
		\end{equation*}
		which implies that
		\begin{equation}\label{corollaryconclusionlp}
		|||P_{1+\eta}u_\varepsilon -\hat u_\varepsilon |||^p_{W^{1,p}(\hat R^\varepsilon)}\leq \hat \rho (\delta),
		\end{equation}
		for $p\geq 2$, where $\hat \rho(\eta)$ is a nonnegative function that tends to zero as $\eta\to 0$.
		
		From Lemma \ref{lemmaestimates}, we have $|||u_\varepsilon-P_{1+\eta}u_\varepsilon|||^p_{W^{1,p}(R^\varepsilon)}\leq c\rho(\delta)$. It follows from \eqref{corollaryconclusionlp} that
		$$
		|||u_\varepsilon-\hat u_\varepsilon|||^p_{W^{1,p}(R^\varepsilon\cap \hat R^\varepsilon)}\leq \tilde \rho(\delta),
		$$
		for $p\geq 2$, where $\tilde \rho(\eta)$ is a nonnegative function that tends to zero as $\eta\to 0$.
		
		For $1<p<2$, we can perform analogous argument to obtain 
		\begin{equation*}
		\begin{gathered}
		\dfrac{c_p}{p(1+\eta)}\left[\dfrac{1}{\varepsilon}\int_{R^\varepsilon}\left|\nabla P_{1+\eta}u_\varepsilon-\nabla  u_\varepsilon\right|^2\left(1+|\nabla P_{1+\eta}u_\varepsilon|+|\nabla u_\varepsilon|\right)^{p-2} dxdy\right.\\
		\left.+\dfrac{1}{\varepsilon}\int_{R^\varepsilon}\left|  P_{1+\eta}u_\varepsilon-  u_\varepsilon\right|^2\left(1+|  P_{1+\eta}u_\varepsilon|+|  u_\varepsilon|\right)^{p-2} dxdy\right] 
		\leq \dfrac{\eta}{1+\eta}\hat V_\varepsilon(\hat u_\varepsilon)+\rho(\delta)^{1/p}
		\end{gathered}
		\end{equation*}
		which gives us
		$$
		|||u_\varepsilon-\hat u_\varepsilon|||^p_{W^{1,p}(R^\varepsilon\cap \hat R^\varepsilon)}\leq \tilde \rho(\delta),
		$$
		where $\tilde \rho(\eta)$ is a nonnegative function which tends to zero as $\eta\to 0$.
	\end{proof}
	\begin{remark}\label{corollarycomparison}
		It follows from \eqref{corollaryconclusionlp} that there exists $\rho:[0,\infty)\mapsto[0,\infty)$ such that
		$$
		|||P_{1+\delta/G_0}u_\varepsilon -\hat u_\varepsilon |||^p_{W^{1,p}(\hat R^\varepsilon)}\leq  \rho (\delta)
		$$
		with $\rho(\delta)\to 0$ as $\delta\to 0$ uniformly in $\varepsilon$ and any piecewise $C^1$ functions $G_\varepsilon$ and $\hat{G}_\varepsilon$ 		uniformly bounded with $\|G_\varepsilon-\hat G_\varepsilon\|_{L^\infty(0,1)}\leq \delta$
		and $f^\varepsilon\in L^{p'}(\mathbb{R}^2)$ satisfying $\|f^\varepsilon\|_{L^{p'}(\mathbb{R}^2)}\leq 1$.
		
	\end{remark}

	\section{The piecewise periodic case}\label{secpiecewise}
	
	Now, we analyze the limit of $\{u_\varepsilon\}_{\varepsilon>0}$ assuming the upper boundary of $R^\varepsilon$ is piecewise periodic. 
	
	More precisely, we assume $G$ satisfies \hyperref[hypG]{(H)} being independent on the first variable in each interval $(\xi_{i-1},\xi_i)\times\mathbb{R}$. 
	We suppose $G$ satisfies $G(x,y)=G_i(y)$ in $x\in I_i=(\xi_{i-1},\xi_i)$ with $G_i(y+L)=G_i(y)$ for all $y\in\mathbb{R}$. Moreover, we assume the function $G_i(\cdot)$ is $C^1$ for all $i=1,\dots,N$ and there exist $0<G_0<G_1$ such that $G_0\leq G_i(\cdot)\leq G_1$ for all $i=1,\dots,N$. 
	
	Notice that the domain $R^\varepsilon$ can now be rewritten as
	\begin{equation}\label{decompositionRepsilon}
	\begin{gathered}
	R^\varepsilon=\left(\bigcup_{i=1}^{N}R_i^\varepsilon\right)\cup \left(\bigcup_{i=1}^{N-1}
	\left\lbrace (\xi_i,y):0<y<\varepsilon\min\{G_{i-1}(\xi_i/\varepsilon),G_i(\xi_i/\varepsilon)\}\right\rbrace \right)
	\end{gathered}
	\end{equation}
	with 
	$$
	R^\varepsilon_i=\left\lbrace (x,y)\in\mathbb{R}: \xi_{i-1}<x<\xi_i,0<y<\varepsilon G_i(x/\varepsilon)\right\rbrace.
	$$
	See Figure \ref{figpw} which illustrates this piecewise periodic thin domain.

	We have the following result.
	\begin{theorem}\label{thmpiecewise}
		Let $u_\varepsilon$ be the solution of problem \eqref{problem} with $f^\varepsilon\in L^{p'}(R^\varepsilon)$ and $\left|\left|\left|f^\varepsilon\right|\right|\right|_{L^{p'}(R^\varepsilon)}\leq c$, for some $c>0$ independent of $\varepsilon>0$. Suppose the function
		$$
		\hat{f}^\varepsilon(x)=\frac{1}{\varepsilon}\int_0^{\varepsilon G\left(x,\frac{x}{\varepsilon}\right)}f(x,y)dy
		$$
		satisfies
		$$
		\hat{f}^\varepsilon\rightharpoonup\hat{f}\mbox{ weakly in }L^{p'}(0,1).
		$$
		
		Then, there exist $u\in W^{1,p}(0,1) $ and $u_1^i\in L^p((\xi_{i-1},\xi_i);W^{1,p}_\#(Y^*_i))$ such that
		\begin{eqnarray*}
			\left\lbrace \begin{array}{llll}
				\mathcal{T}^i_\varepsilon u_\varepsilon\rightharpoonup u\mbox{ weakly in } L^p((\xi_{i-1},\xi_i);W^{1,p}(Y^*_i)),\\
				\mathcal{T}^i_\varepsilon \left(\partial_x u_\varepsilon\right)\rightharpoonup \partial_xu+\partial_{y_1}u_1^i(x,y_1,y_2)\mbox{ weakly in } L^p\left((\xi_{i-1},\xi_i);W^{1,p}(Y^*_i)\right),\\
				\mathcal{T}^i_\varepsilon \left(\partial_y u_\varepsilon\right)\rightharpoonup \partial_{y_2}u_1^i(x,y_1,y_2)\mbox{ weakly in } L^p\left((\xi_{i-1},\xi_i);W^{1,p}(Y^*_i)\right)
			\end{array}\right.
		\end{eqnarray*}
		and $u$ is the solution of the problem
		\begin{equation}\label{homogenizedlimit}
		\int_0^1 \left\lbrace q(x)|u'|^{p-2}u'\varphi' +r(x)\,|u|^{p-2}u\varphi \right\rbrace dx=\int_0^1\hat{f}\varphi dx, \quad \varphi\in W^{1,p}(0,1),
		\end{equation}
		where  $q$, $r:(0,1)\to\mathbb{R}$ are piecewise constant functions given by
		$$
		\begin{gathered}
		q(x)=q_i, \quad \textrm{ if } x \in (\xi_{i-1}, \xi_i),\\
		r(x)=r_i, \quad \textrm{ if } x \in (\xi_{i-1}, \xi_i),\\
		\end{gathered}
		$$
		with $r_i$ and $q_i$ are given by
		\begin{equation}\label{qcoefsdefinition}
		\begin{gathered}
		q_i=\int_{Y^*_i}|\nabla v^i|^{p-2}\partial_{y_1} v^i \, dy_1dy_2 \\
		r_i=\dfrac{|Y^*_i|}{L}
		\end{gathered}
		\end{equation}
		where $v^i$ is the solution of the auxiliary problem
		\begin{eqnarray}\label{auxiliarproblem}
		\begin{gathered}
		\displaystyle\int_{Y^*_i}\left|\nabla v^i\right|^{p-2}\nabla v^i\nabla \psi dy_1dy_2=0, \quad \forall\psi\in W^{1,p}_{\#}(Y^*_i),\quad \left\langle \psi \right\rangle_{Y^{*}_{i}}=0\\
		(v^i-y_1)\in  W^{1,p}_{\#}(Y^*_i),\quad \left\langle v-y_{1} \right\rangle_{Y^{*}_{i}}=0.
		\end{gathered}
		\end{eqnarray} 
	\end{theorem}
	\begin{proof}
		
		First, by \eqref{decompositionRepsilon}, we can rewrite \eqref{variationalproblem} taking account the partition $\{\xi_i\}_{i=1}^N$ as
		\begin{equation}\label{variationalproblempartition}
		\sum_{i=1}^{N}\int_{R^\varepsilon_i} \left\{ |\nabla u_\varepsilon|^{p-2}\nabla u_\varepsilon\nabla \varphi +|u_\varepsilon|^{p-2}u_\varepsilon\varphi \right\}dxdy=\int_{R^\varepsilon} f^\varepsilon\varphi \, dxdy, \qquad \varphi \in W^{1,p}(R^\varepsilon).
		\end{equation}
		
		Hence, using \eqref{variationalproblempartition} and Proposition \ref{unfoldingproperties}, we obtain from \eqref{variationalproblem} with test functions $\varphi(x,y)=\varphi_i(x) \in C_0^\infty(\xi_{i-1},\xi_i)$ that 
		\begin{equation} \label{variationalunfolded}
		\begin{gathered}
		\int_{(\xi_{i-1},\xi_i)\times Y^*_i}\mathcal{T}^i_\varepsilon\left(\left|\nabla u_\varepsilon\right|^{p-2}\nabla u_\varepsilon\right)\mathcal{T}^i_\varepsilon\nabla\varphi_i dxdy_1dy_2+\frac{L}{\varepsilon}\int_{R_{1i}^\varepsilon}\left|\nabla u_\varepsilon\right|^{p-2}\nabla u_\varepsilon\nabla\varphi_i dxdy\\
		+\int_{(\xi_{i-1},\xi_i)\times Y^*_i}\mathcal{T}^i_\varepsilon\left(\left| u_\varepsilon\right|^{p-2} u_\varepsilon\right)\mathcal{T}^i_\varepsilon\varphi_i dxdy_1dy_2+\frac{L}{\varepsilon}\int_{R_{1i}^\varepsilon}\left| u_\varepsilon\right|^{p-2} u_\varepsilon\varphi_i dxdy
		= \dfrac{L}{\varepsilon}\int_{R^\varepsilon}f^\varepsilon  \varphi_i dxdy.
		\end{gathered}
		\end{equation}
		
		By Proposition \ref{uniformlimtation} and Theorem \ref{thmressonant}, there exist 
		$$u^i\in W^{1,p}(\xi_{i-1},\xi_i), \quad u_1^i\in L^p((\xi_{i-1},\xi_i);W^{1,p}_\#(Y^*_i)) \quad \textrm{ and } \quad a_0^i\in L^p\left((\xi_{i-1},\xi_i)\times Y^*_i\right)^2$$ 
		such that, up to subsequences,
		\begin{eqnarray}\label{unfoldingconvergence}
		\left\lbrace \begin{array}{llll}
		\mathcal{T}^i_\varepsilon u_\varepsilon\rightarrow u^i\mbox{ strongly in } L^p\left((\xi_{i-1},\xi_i)\times Y^*_i\right),\\
		\mathcal{T}^i_\varepsilon \left(\partial_x u_\varepsilon\right)\rightharpoonup \partial_xu^i+\partial_{y_1}u_1^i(x,y_1,y_2)\mbox{ weakly in } L^p\left((\xi_{i-1},\xi_i);W^{1,p}(Y^*_i)\right),\\
		\mathcal{T}^i_\varepsilon \left(\partial_y u_\varepsilon\right)\rightharpoonup \partial_{y_2}u_1^i(x,y_1,y_2)\mbox{ weakly in } L^p\left((\xi_{i-1},\xi_i);W^{1,p}(Y^*_i)\right),\\
		\mathcal{T}^i_\varepsilon\left(\left|\nabla u_\varepsilon\right|^{p-2}\nabla u_\varepsilon\right)\rightharpoonup a_0^i\mbox{ weakly in } L^p\left((\xi_{i-1},\xi_i)\times Y^*_i\right)^2.
		\end{array}\right.
		\end{eqnarray}
		
		We still have from Remark \ref{re483} that
		$$
		\left|\mathcal{T}^i_\varepsilon u_\varepsilon\right|^{p-2}\mathcal{T}^i_\varepsilon u_\varepsilon\rightarrow |u^i|^{p-2}u^i\mbox{ strongly in } L^{p'}\left((\xi_{i-1},\xi_i)\times Y^*_i\right).
		$$
		Then, we can pass to the limit in \eqref{variationalunfolded} getting 
		\begin{equation}\label{prehomogenized}
		\dfrac{1}{L}\int_{(\xi_{i-1},\xi_i)\times Y^*_i}a_0^i\nabla\varphi_i +\left|u^i\right|^{p-2}u^i\varphi_i dxdy_1dy_2=\int_0^1 \hat{f}\varphi_i dx
		\quad \textrm{ for each } i =1, ..., N.
		\end{equation}
		
		Now take $\phi\in C_0^\infty\left(\xi_{i-1},\xi_i\right)$ and $\psi\in W^{1,p}_\#(Y^*_i)$. Extend $\psi$ periodically in the variable $y_1$ and define the sequence
		\begin{equation*}
		v_\varepsilon(x,y)=\varepsilon\phi(x)\psi\left(\frac{x}{\varepsilon},\frac{y}{\varepsilon}\right).
		\end{equation*}
		We have
		\begin{eqnarray*}
			\mathcal{T}^i_\varepsilon v_\varepsilon\rightarrow 0\mbox{ strongly in }L^p\left((\xi_{i-1},\xi_i)\times Y^*_i\right),\\
			\mathcal{T}^i_\varepsilon(\partial_x v_\varepsilon)\rightarrow \phi \partial_{y_1}\psi \mbox{ strongly in }L^p\left((\xi_{i-1},\xi_i)\times Y^*_i\right),\\
			\mathcal{T}^i_\varepsilon(\partial_y v_\varepsilon)\rightarrow \phi \partial_{y_2}\psi \mbox{ strongly in }L^p\left((\xi_{i-1},\xi_i)\times Y^*_i\right).
		\end{eqnarray*}
		
		Thus, taking $v_\varepsilon$ as a test function in \eqref{variationalunfolded}, we obtain at $\varepsilon=0$ that 
		\begin{eqnarray}\label{preaxiliar}
		\int_{(\xi_{i-1},\xi_i)\times Y^*_i}a_0^i\phi(x)\nabla_y\psi dxdy_1dy_2=0.
		\end{eqnarray}
		
		Hence, we get from \eqref{preaxiliar} and the density of the tensor product 
		$\quad C_0^\infty(\xi_{i-1},\xi_i) \otimes W^{1,p}_\#(Y^*_i)$ in $L^p((\xi_{i-1},\xi_i);W^{1,p}_\#(Y^*_i))$ that
		\begin{eqnarray}\label{preaxiliar01}
		\int_{(\xi_{i-1},\xi_i)\times Y^*_i}a_0^i \, \nabla_y\psi \, dxdy_1dy_2=0, \quad \forall\psi\in L^p((\xi_{i-1},\xi_i);W^{1,p}_\#(Y^*_i)).
		\end{eqnarray}
		
		Now, let us identify $a_0^i$ given by \eqref{unfoldingconvergence} for each $i$. 
		For this sake, take $u_1\in L^p((\xi_{i-1},\xi_i);W^{1,p}_\#(Y^*_i))$ and $u^i_1\in W^{1,p}(\xi_{i-1},\xi_i)$ 
		given by \eqref{unfoldingconvergence}. Extend $\nabla_y u_1^i$ periodically in the $y_1$-direction, and then define 
		\begin{equation*} \label{eqW_}
		W_\varepsilon(x,y)=\left(\partial_xu^i(x),0\right)+\nabla_y u_1^i\left(x,\frac{x}{\varepsilon},\frac{y}{\varepsilon}\right),\quad (x,y)\in R^{\varepsilon}_{i}.
		\end{equation*}
		
		Notice that $W_\varepsilon\in L^p(R^\varepsilon_i)\times L^p(R^\varepsilon_i)$, and due to  Proposition \ref{convergenceplaplacetype}, we have
		\begin{equation}\label{convWeps}
		\begin{gathered}
		\mathcal{T}^i_\varepsilon W_\varepsilon\rightarrow \left(\partial_xu^i,0\right)+\nabla_y u_1^i, \quad \textrm{ and } \\
		\mathcal{T}^i_\varepsilon\left(|W_\varepsilon|^{p-2}W_\varepsilon\right)\rightarrow \left|\left(\partial_xu^i,0\right)+\nabla_y u_1^i\right|^{p-2}\left[\left(\partial_xu^i,0\right)+\nabla_y u_1^i\right]   
		\end{gathered}
		\end{equation}
		strongly in $L^p\left((\xi_{i-1},\xi_i)\times Y^*_i\right)^2$ and $L^{p'}\left((\xi_{i-1},\xi_i)\times Y^*_i\right)^2$ respectively.
		Moreover, we can see that the right hand side of the inequality 
		\begin{equation} \label{693}
		0\leq\int_{(\xi_{i-1},\xi_i)\times Y^*_i}\mathcal{T}^i_\varepsilon\left[|\nabla u_\varepsilon|^{p-2}\nabla u_\varepsilon-|W_\varepsilon|^{p-2}W_\varepsilon\right]\mathcal{T}^i_\varepsilon\left(\nabla u_\varepsilon-W_\varepsilon\right)dxdy_1dy_2
		\end{equation}
		converges to zero as $\varepsilon \to 0$. 
		Notice that by the monotonicity of $|\cdot|^{p-2}\cdot$ (see Proposition \ref{proposicaoplaplaciano}) inequality \eqref{693} is obtained. To pass the limit in \eqref{693}, we evaluate each term of the integral. 
		Using \eqref{variationalunfolded}, \eqref{unfoldingconvergence} and denoting $dY=dy_1 dy_2$, we get that 
		\begin{equation*}
		\begin{gathered}
		\lim_{\varepsilon\rightarrow 0} \int_{(\xi_{i-1},\xi_i)\times Y^*_i}\mathcal{T}^i_\varepsilon\left(|\nabla u_\varepsilon|^{p-2}\nabla u_\varepsilon\right)\mathcal{T}^i_\varepsilon(\nabla u_\varepsilon)dxdY  \\
		= \lim_{\varepsilon\rightarrow 0}\left[\int_{(\xi_{i-1},\xi_i)\times Y^*_i}\mathcal{T}^i_\varepsilon f^\varepsilon \mathcal{T}^i_\varepsilon u_\varepsilon dxdY+\dfrac{L}{\varepsilon}\int_{R^\varepsilon_1}f^\varepsilon u_\varepsilon dxdy \right.\\
		\left. - \int_{(\xi_{i-1},\xi_i)\times Y^*_i}\mathcal{T}^i_\varepsilon\left(| u_\varepsilon|^{p-2} u_\varepsilon\right)\mathcal{T}^i_\varepsilon u_\varepsilon dxdY -\dfrac{L}{\varepsilon}\int_{R^\varepsilon_1}|u_\varepsilon|^{p-2} u_\varepsilon u_\varepsilon dxdy\right]\\
		=\int_{(\xi_{i-1},\xi_i)\times Y^*_i}\left(\hat{f}-|u^i|^{p-2}u^i\right)u^idxdY.
		\end{gathered}
		\end{equation*}
		Consequently, we get from \eqref{prehomogenized} that 
		\begin{equation} \label{eq777}
		\lim_{\varepsilon\rightarrow 0} \int_{(\xi_{i-1},\xi_i)\times Y^*_i}\mathcal{T}^i_\varepsilon\left(|\nabla u_\varepsilon|^{p-2}\nabla u_\varepsilon\right)\mathcal{T}^i_\varepsilon(\nabla u_\varepsilon)dxdY =
		\int_{(\xi_{i-1},\xi_i)\times Y^*_i} a_0 \nabla u^i dxdY.
		\end{equation}
		
		On the other hand, due to \eqref{unfoldingconvergence}, \eqref{convWeps} and \eqref{preaxiliar01}, we get
		\begin{equation}  \label{eq780}
		\begin{gathered}
		\lim_{\varepsilon\rightarrow 0}  \int_{(\xi_{i-1},\xi_i)\times Y^*_i}\mathcal{T}^i_\varepsilon\left(|\nabla u_\varepsilon|^{p-2}\nabla u_\varepsilon\right)\mathcal{T}^i_\varepsilon(W_\varepsilon)dxdY \\ =  \int_{(\xi_{i-1},\xi_i)\times Y^*_i} a_0^i \left(\nabla u^{i}+\nabla_yu_1\right)  dxdY  
		= \int_{(\xi_{i-1},\xi_i)\times Y^*_i} a_0^i\nabla u^i dxdY,
		\end{gathered}
		\end{equation}
		since $u_1^i\in L^p\left((\xi_{i-1},\xi_i);W^{1,p}_\#(Y^*_i)\right)$.
		
		Finally, we have 			
		\begin{equation} \label{eq790}
		\begin{gathered}
		\lim_{\varepsilon\rightarrow 0}   \int_{(\xi_{i-1},\xi_i)\times Y^*_i}\mathcal{T}^i_\varepsilon\left(|W_\varepsilon|^{p-2}W_\varepsilon\right)\mathcal{T}^i_\varepsilon\left(\nabla u_\varepsilon-W_\varepsilon\right)dxdY = 0,
		\end{gathered}
		\end{equation}
		by \eqref{convWeps} and \eqref{unfoldingconvergence}. 
		Indeed, we have $\mathcal{T}^i_\varepsilon (\nabla u_\varepsilon - W_\varepsilon) \rightharpoonup 0$ weakly in $L^p((\xi_{i-1},\xi_i)\times Y^*_i)$.
		
		Thus, from \eqref{eq777}, \eqref{eq780} and \eqref{eq790}, we can pass to the limit in \eqref{693} to get
		\begin{equation}\label{694}
		\int_{(\xi_{i-1},\xi_i)\times Y^*_i}\mathcal{T}^i_\varepsilon\left[|\nabla u_\varepsilon|^{p-2}\nabla u_\varepsilon-|W_\varepsilon|^{p-2}W_\varepsilon\right]\mathcal{T}^i_\varepsilon\left(\nabla u_\varepsilon-W_\varepsilon\right)dxdY\rightarrow 0.
		\end{equation}
		
		Next, suppose $p\geq 2$. By Proposition \ref{proposicaoplaplaciano} and \eqref{694}, we have 
		\begin{eqnarray*} \label{eq823}
		&& \int_{(\xi_{i-1},\xi_i)\times Y^*_i}\left|\mathcal{T}^i_\varepsilon \nabla u_\varepsilon-\mathcal{T}^i_\varepsilon W_\varepsilon\right|^p dxdY\nonumber\\
		&\leq&c \int_{(\xi_{i-1},\xi_i)\times Y^*_i}\mathcal{T}^i_\varepsilon\left(|\nabla u_\varepsilon|^{p-2}\nabla u_\varepsilon-|W_\varepsilon|^{p-2}W_\varepsilon\right)\left(\mathcal{T}^i_\varepsilon \nabla u_\varepsilon-\mathcal{T}^i_\varepsilon W_\varepsilon\right) dxdY\nonumber\\
		&\rightarrow& 0 \qquad \mbox{ as }\varepsilon\rightarrow 0.
		\end{eqnarray*}
		
		Now, suppose $1<p\leq2$. Then, 
		\begin{eqnarray*}\label{824}
		&&\int_{(\xi_{i-1},\xi_i)\times Y^*_i}\left|\mathcal{T}^i_\varepsilon \nabla u_\varepsilon-\mathcal{T}^i_\varepsilon W_\varepsilon\right|^p dxdY\nonumber\\
		&=&\int_{(\xi_{i-1},\xi_i)\times Y^*_i}\left|\mathcal{T}^i_\varepsilon \nabla u_\varepsilon-\mathcal{T}^i_\varepsilon W_\varepsilon\right|^p \frac{\left(1+|\mathcal{T}^i_\varepsilon\nabla u_\varepsilon|+|\mathcal{T}^i_\varepsilon W_\varepsilon|\right)^{\frac{(p-2)p}{2}}}{\left(1+|\mathcal{T}^i_\varepsilon\nabla u_\varepsilon|+|\mathcal{T}^i_\varepsilon W_\varepsilon|\right)^{\frac{(p-2)p}{2}}} dxdY.
		\end{eqnarray*}
		Hence, using a H\"older's inequality for the exponent $\frac{2}{p}$ (and its conjugate $\frac{2}{2-p}$) and Proposition \ref{proposicaoplaplaciano}, 
		\begin{eqnarray}\label{aob2}
		&&\int_{(\xi_{i-1},\xi_i)\times Y^*_i}\left|\mathcal{T}^i_\varepsilon \nabla u_\varepsilon-\mathcal{T}^i_\varepsilon W_\varepsilon\right|^p dxdY\nonumber\\
		&\leq& \left[\int_{(\xi_{i-1},\xi_i)\times Y^*_i}\left|\mathcal{T}^i_\varepsilon\nabla u_\varepsilon-\mathcal{T}^i_\varepsilon W_\varepsilon\right|^{2}\left(1+|\mathcal{T}^i_\varepsilon\nabla u_\varepsilon|+|\mathcal{T}^i_\varepsilon W_\varepsilon|\right)^{p-2}dxdY\right]^{p/2}\nonumber\\
		&&\quad\quad\quad \cdot\left[\int_{(\xi_{i-1},\xi_i)\times Y^*_i}\left(1+|\mathcal{T}^i_\varepsilon\nabla u_\varepsilon|+|\mathcal{T}^i_\varepsilon W_\varepsilon|\right)^{p}dxdY\right]^{(2-p)/2}\nonumber\\
		&\leq&c \int_{(\xi_{i-1},\xi_i)\times Y^*_i}\mathcal{T}^i_\varepsilon\left(|\nabla u_\varepsilon|^{p-2}\nabla u_\varepsilon-|W_\varepsilon|^{p-2}W_\varepsilon\right)\left(\mathcal{T}^i_\varepsilon \nabla u_\varepsilon-\mathcal{T}^i_\varepsilon W_\varepsilon\right) dxdY.\nonumber
		\end{eqnarray}
		Consequently, as $\varepsilon\to 0$, one gets for any $p>1$ that 
		\begin{equation}\label{correctorthm0001}
		\int_{(\xi_{i-1},\xi_i)\times Y^*_i}\left|\mathcal{T}^i_\varepsilon \nabla u_\varepsilon-\mathcal{T}^i_\varepsilon W_\varepsilon\right|^p dxdY\to 0.
		\end{equation}

		Now, we prove 
		$$
		\int_{(\xi_{i-1},\xi_i)\times Y^*_i}\mathcal{T}^i_\varepsilon\left(|\nabla u_\varepsilon|^{p-2}\nabla u_\varepsilon-|W_\varepsilon|^{p-2}W_\varepsilon\right)\varphi \, dxdY \rightarrow 0
		$$
		for any test function $\varphi\in C_0^\infty\left((\xi_{i-1},\xi_i)\times Y^*_i\right)\times C_0^\infty\left((\xi_{i-1},\xi_i)\times Y^*_i\right)$.
		
		Let $p\geq 2$. Therefore, from Corollary \eqref{corolarioplaplaciano} and H\"older's inequality, we get
		\begin{eqnarray}\label{eq609}
		&&\nonumber \int_{(\xi_{i-1},\xi_i)\times Y^*_i}\mathcal{T}^i_\varepsilon\left(|\nabla u_\varepsilon|^{p-2}\nabla u_\varepsilon-|W_\varepsilon|^{p-2}W_\varepsilon\right)\varphi dxdY\\\nonumber\nonumber
		&\leq& c \int_{(\xi_{i-1},\xi_i)\times Y^*_i}\left(1+|\mathcal{T}^i_\varepsilon \nabla u_\varepsilon|+|\mathcal{T}^i_\varepsilon W_\varepsilon|\right)^{p-2}\left|\mathcal{T}^i_\varepsilon \left(\nabla u_\varepsilon-W_\varepsilon\right)\right|dxdY\\\nonumber
		&\leq&\left[\int_{(\xi_{i-1},\xi_i)\times Y^*_i}\left(1+|\mathcal{T}^i_\varepsilon \nabla u_\varepsilon|+|\mathcal{T}^i_\varepsilon W_\varepsilon|\right)^{p}dxdY\right]^{1/p'} 
		\cdot \left[\int_{(\xi_{i-1},\xi_i)\times Y^*_i}\left|\mathcal{T}^i_\varepsilon \nabla u_\varepsilon-\mathcal{T}^i_\varepsilon W_\varepsilon\right|^p dxdY\right] ^{1/p}.
		\end{eqnarray}
		
		For $1<p<2$, we perform analogous arguments. Using Corollary \ref{corolarioplaplaciano} and H\"older's inequality, one gets 
		\begin{eqnarray}\label{eq617}
		&&\nonumber  \int_{(\xi_{i-1},\xi_i)\times Y^*_i}\mathcal{T}^i_\varepsilon\left(|\nabla u_\varepsilon|^{p-2}\nabla u_\varepsilon-|W_\varepsilon|^{p-2}W_\varepsilon\right)\varphi dxdY\\\nonumber
		&\leq&c  \int_{(\xi_{i-1},\xi_i)\times Y^*_i}\left|\mathcal{T}^i_\varepsilon\nabla u_\varepsilon-\mathcal{T}^i_\varepsilon W_\varepsilon\right|^{p-1}dxdY
		=  c\left[\int_{(\xi_{i-1},\xi_i)\times Y^*_i}\left|\mathcal{T}^i_\varepsilon\nabla u_\varepsilon-\mathcal{T}^i_\varepsilon W_\varepsilon\right|^{p}dxdY\right]^{1/p'}.
		\end{eqnarray}	
		Therefore, for any $p>1$, we get from \eqref{correctorthm0001}, \eqref{eq609} and \eqref{eq617} that
		\begin{equation}
		\int_{(\xi_{i-1},\xi_i)\times Y^*_i}\left[ a_0^i-a_p\left((\partial_xu^i,0)+\nabla_y u_1^i\right)\right]\varphi dx dY=0,
		\end{equation}	
		for any $\varphi\in C_0^\infty\left((\xi_{i-1},\xi_i)\times Y^*_i\right)\times C_0^\infty\left((\xi_{i-1},\xi_i)\times Y^*_i\right)$ and $i=1,\cdots,N$.
		
		Now, let us associate $a_0^i$ with the auxiliary problem \eqref{auxiliarproblem}. 
		We first rewrite \eqref{preaxiliar01} as 
		\begin{eqnarray}\label{preaxiliar02}
		\int_{(\xi_{i-1},\xi_i)\times Y^*_i}\left|\nabla u^i+\nabla_yu_1^i\right|^{p-2}\left(\nabla u^i+\nabla_yu_1^i\right)\nabla_y\psi \, dxdY=0,
		\end{eqnarray}
		for any $\psi \in L^p((\xi_{i-1},\xi_i);W^{1,p}_\#(Y^*_i))$.
		
		From Minty-Browder Theorem, one can prove that \eqref{preaxiliar02} sets a well posed problem in the following sense: for each $u \in W^{1,p}(\xi_{i-1},\xi_i)$, \eqref{preaxiliar02} possesses an unique solution $u_1 \in L^p((\xi_{i-1},\xi_i);W^{1,p}_\#(Y^*_i)/\mathbb{R})$. Notice that $W^{1,p}_\#(Y^*_i)/\mathbb{R}$ is identified with the closed subspace of $W^{1,p}_\#(Y^*_i)$ consisting of all its functions with zero average.

		Multiplying the solution $v^i$ of the equation \eqref{auxiliarproblem} by $(u^i)'$ we obtain a function $(u^i)'v$ which depends on $(x,y_1,y_2)$ and belongs to the space $L^p((\xi_{i-1},\xi_i);W^{1,p}_\#(Y^*_i))/\mathbb{R})$. Next, multiplying \eqref{auxiliarproblem} by $|\partial_xu^i|^{p-2}\partial_xu^i$ and $\phi\in C_0^\infty(\xi_{i-1},\xi_i)$, and integrating in $(\xi_{i-1},\xi_i)$,  we get
		\begin{equation*}
		\int_{(\xi_{i-1},\xi_i)\times Y^*_i} \phi \, \left|\partial_xu^i\nabla_yv^i\right|^{p-2}\partial_xu^i\nabla_yv^i\,  \nabla_y\varphi \, dxdY=0, \quad  
		\forall\varphi\in W^{1,p}_\#(Y^*_i)/\mathbb{R}.
		\end{equation*}
		Thus, from the density of tensor product $C_0^\infty(\xi_{i-1},\xi_i)\otimes (W^{1,p}_\#(Y^*_i)/\mathbb{R})$, we get
		\begin{equation}\label{preauxiliar03}
		\int_{(\xi_{i-1},\xi_i)\times Y^*_i}\left|\partial_xu^i\nabla_yv^i\right|^{p-2}\partial_xu^i\nabla_yv^i\nabla_y\psi dxdY=0, \quad \forall\psi\in L^p((\xi_{i-1},\xi_i);W^{1,p}_\#(Y^*_i)/\mathbb{R}).
		\end{equation}
		
		Hence, from equations \eqref{preaxiliar02} and \eqref{preauxiliar03}, we get, by uniqueness, that 
		$$
		\partial_xu^i(x)\nabla_yv^i(y_1,y_2)=(\partial_x u^i(x),0)+\nabla_yu_1^i(x,y_1,y_2) \mbox{ a.e. in }(\xi_{i-1},\xi_i)\times Y^*_i.
		$$
		
		Moreover, taking test functions $\varphi\in C_0^\infty\left(\displaystyle\cup_{i=1}^{N}(\xi_{i-1},\xi_i)\right)$ in \eqref{variationalproblempartition}, we can use Proposition \ref{unfoldingproperties} and \eqref{prehomogenized} in each interval $(\xi_{i-1},\xi_i)$ to obtain 
		$$
		\sum_{i=1}^{N} \dfrac{1}{L}\int_{(\xi_{i-1},\xi_i)\times Y^*_i}\left|\partial_xu^i\nabla_yv^i\right|^{p-2}\partial_xu^i\nabla_yv^i\nabla\varphi +\left|u^i\right|^{p-2}u^i\varphi dxdY=\int_0^1 \hat{f}\varphi dx,
		$$
		which is equivalent to
		\begin{equation}\label{variationalpiecewise}
		\sum_{i=1}^{N} \int_{\xi_{i-1}}^{\xi_i} \left[q^i|\partial_xu^i|^{p-2}\partial_xu^i\partial_x\varphi+\dfrac{|Y^*_i|}{L}\left|u^i\right|^{p-2}u^i\varphi\right] dx=	\int_0^1 \hat{f}\varphi dx.
		\end{equation}
		
		Then, using $q^i$ and  $r^i$ are given by \eqref{qcoefsdefinition}, one can obtain the following limit problem
		$$
		\int_0^1 \left[q(x)|u'|^{p-2}u'\varphi'+ r(x)|u|^{p-2}u\varphi \right]dx =\int_0^1\hat{f}\varphi dx,\qquad\forall \varphi\in W^{1,p}(0,1),
		$$
		which has a unique solution $u\in W^{1,p}(0,1)$, by Minty-Browder's Theorem. 
		Thus, from \eqref{variationalpiecewise}, we get 
		$$
		u(x)=u^i(x) \mbox{ a.e. in } (\xi_{i-1},\xi_i)
		$$
		concluding the proof since $q_i>0$ for each $i$. 
		Indeed, by \eqref{auxiliarproblem}, we can take $(v^i-y_1)\in W^{1,p}_{\#,0}(Y^*)$ as a test function, obtaining 
		$$
		\begin{gathered}
		q^i 
		= \dfrac{1}{L}\int_{Y^*_i}|\nabla v^i|^{p-2}\nabla v^i\left((1,0)+\nabla v^i-(1,0)\right) dy_1dy_2=\dfrac{1}{L}\int_{Y^*_i}|\nabla v^i|^p dy_1dy_2>0.
		\end{gathered}
		$$ 
	\end{proof}

\section{The locally periodic case}\label{secgencase}

	In this section, we provide the proof of our main result, Theorem \ref{mainthm}.
%
	\begin{proof}[Proof of Theorem \ref{mainthm}]
		Using Proposition \ref{uniformlimtation} and Theorem \ref{locallyperiodicconvergence}, there is $u_0\in W^{1,p}(0,1)$ such that, up to subsequences,
		\begin{equation}\label{1160}
		T_\varepsilon^{lp}u_\varepsilon\rightharpoonup \chi u_0\mbox{ weakly in }L^p\left((0,1)\times (0,L)\times (0,G_1)\right),
		\end{equation}
		where $\chi$ is the characteristic function of $(0,1)\times Y^*(x)$.
		
		We show that $u_0$ satisfies the Neumann problem \eqref{homogenizedlimit}. To do this, we use a kind of discretization argument on the oscillating thin domains. 
		We first proceed as in \cite[Theorem 2.3]{AP10} fixing a parameter $\delta>0$ in order to set a function $G^\delta(x,y)$ with the property 
		$0\leq G^\delta(x,y)-G(x,y)\leq \delta$ in $(0,1)\times \mathbb{R}$ and such that the function 
		$G^\delta$ satisfies \hyperref[hypG]{(H)} and is piecewise periodic. 
		
		Let us construct this function. Recall that $G$ is uniformly $C^1$ in each of the domains $(\xi_{i-1},\xi_i)\times \mathbb{R}$. Also, it is periodic in the second variable. In particular, for $\delta>0$ small enough and for a fixed $z\in (\xi_{i-1},\xi_i)$ we have that there exists a small interval $(z-\eta,z+\eta)$ with $\eta$ depending only on $\delta$ such that $|G(x,y)-G(z,y)|+|\partial_y G(x,y) - \partial_y G(z,y)|<\delta/2$ for all $x\in (z-\eta,z+\eta)\cap (\xi_{i-1},\xi_i)$ and for all $y\in\mathbb{R}$. This allows us to select a finite number of points: $\xi_{i-1}=\xi_{i-1}^1<\xi_{i-1}^2<\cdots<\xi_{i-1}^r=\xi_i$ with $\xi_{i-1}^r-\xi_{i-1}^{r-1}<\eta$ in such way that $G^\delta(x,y) = G(\xi^r_{i-1},y) + \delta/2$ defined for $x \in (\xi^r_{i-1},\xi^{r+1}_{i-1})$ and $y \in \mathbb{R}$ satisfies $|\partial_y G^\delta(x,y) - \partial_y G(z,y)|\leq \delta$ in $(\xi_{i-1}^r,\xi_{i-1}^{r+1})\times\mathbb{R}$. Notice that this construction can be done for all $i=1,\dots,N$.
		In particular, if we rename all the constructed points $\xi_i^k$ by $0=z_0<z_1<\dots<z_m=1$, 
		for some $m = m(\delta)$, we get that $G^\delta(x,y)=G_i^\delta(y)$ for $(x,y)\in (z_{i-1},z_i)\times \mathbb{R}$ and $i=1,\dots,m$ 
		is a piecewise $C^1$-function which is $L$-periodic in the second variable $y$.
		
		Finally, we set $G^\delta_\varepsilon(x)=G^\delta(x,x/\varepsilon)$ considering the following domains 
		\begin{equation}
		\begin{gathered}
		R^{\varepsilon,\delta}=\{(x,y):x\in(0,1),0<y<\varepsilon G_\varepsilon^\delta(x) \}.
		\end{gathered}
		\end{equation}
		It follows from Theorem \ref{thmpiecewise} that, for each $\delta>0$ fixed, there exist $u^\delta\in W^{1,p}(0,1)$ and 
		$u_1^{i,\delta}\in  L^p((\xi_{i-1},\xi_i);W^{1,p}_\#(Y^*_i))$ in such way that the
		solutions $u_{\varepsilon,\delta}$ of \eqref{problem} in $R^{\varepsilon,\delta}$ satisfy 
		\begin{equation}\label{convepsilondelta}
		\left\lbrace \begin{array}{llll}
		\mathcal{T}^\delta_\varepsilon u_{\varepsilon,\delta}\to u^\delta\mbox{ strongly in } L^p((z_{i-1},z_i);W^{1,p}(Y^*_i)),\\
		\mathcal{T}^\delta_\varepsilon \left(\partial_x u_{\varepsilon,\delta}\right)\rightharpoonup \partial_xu^\delta+\partial_{y_1}u_1^{i,\delta}(x,y_1,y_2)\mbox{ weakly in } L^p\left((z_{i-1},z_i);W^{1,p}(Y^*_i)\right),\\
		\mathcal{T}^\delta_\varepsilon \left(\partial_y u_{\varepsilon,\delta}\right)\rightharpoonup \partial_{y_2}u_1^{i,\delta}(x,y_1,y_2)\mbox{ weakly in } L^p\left((z_{i-1},z_i);W^{1,p}(Y^*_i)\right),\\
		\mathcal{T}^\delta_\varepsilon(\left|\nabla u_{\varepsilon,\delta}\right|^{p-2}\nabla u_{\varepsilon,\delta})\rightharpoonup q^\delta a_p(\partial_x u^\delta)\mbox{ weakly in } L^p\left((z_{i-1},z_i)\times Y^*_i\right)^2. 
		\end{array}\right.
		\end{equation}
		Also, we have that $u^\delta$ is the solution of the Neumann problem
		\begin{equation} \label{quasihomogenized}
		\int_0^1 \left\lbrace q^\delta(x)|(u^\delta)'|^{p-2}(u^\delta) '\varphi' +r^\delta(x)\,|u^\delta|^{p-2}u^\delta\varphi \right\rbrace dx 
		=\int_0^1\hat{f}\varphi dx, \quad \forall \varphi\in W^{1,p}(0,1),
		\end{equation}
		with 
		\begin{equation}\label{pqdefinition}
		\begin{gathered}
		q^\delta(x)=\dfrac{1}{L}\sum_{i=1}^{N-1}\chi_{I_i}(x)\int_{Y^*_i}|\nabla v^i|^{p-2}\partial_{y_1} v^i \, dy_1dy_2
		\quad \textrm{ and } \quad 
		r^\delta(x)=\sum_{i=1}^{N-1}\chi_{I_i}(x)\dfrac{|Y^*_i|}{L}.
		\end{gathered}
		\end{equation}
		$\chi_{I_i}$ is the characteristic function of $(\xi_{i-1},\xi_i)$ and $v^i$ is the solution of \eqref{auxiliarproblem} in $Y^*_i$ which is given by
		\begin{equation*} 
		Y^*_{i}=\left\lbrace \right(y_1,y_2) \in \mathbb{R}^2 : 0<y_1<L\mbox{ and }0<y_2<G_{i}(y_1)\rbrace.
		\end{equation*}

		Now, we pass to the limit in \eqref{quasihomogenized} as $\delta\to 0$. From Lemma \ref{lemmaappendix}, we have the uniform convergence of $q^\delta$ and $r^\delta$ to $q$ and $r$ where
		\begin{equation}\label{pqdefinitionlimit}
		\begin{gathered}
		q(x)=\dfrac{1}{L}\int_{Y^*(x)}|\nabla v|^{p-2}\partial_{y_1} v \, dy_1dy_2 
		\quad \textrm{ and } \quad 
		r(x)= \dfrac{|Y^*(x)|}{L}.
		\end{gathered}
		\end{equation}
		
		Notice that $q(x)>0$.
		Furthermore, we have the solutions $u^\delta \in W^{1,p}(0,1)$ of \eqref{quasihomogenized} are uniformly bounded in $\delta$. 
		Thus, there exists $u^*\in W^{1,p}(0,1)$ such that 
		$u^\delta\rightharpoonup u^*$ weakly in $W^{1,p}(0,1)\mbox{ and strongly in }L^p(0,1)$.
		Indeed, we have the strong convergence 
		\begin{equation}\label{convfinalthm}
		u^\delta\to u^*\mbox{ in }W^{1,p}(0,1).
		\end{equation}	
		
		To prove this, we use the following norm
		$$
		\|\cdot\|_{L^p_\delta(0,1)}^p=\int_0^1q^\delta |\cdot|^p dx.
		$$
		By Proposition \ref{proposicaoplaplaciano} and \eqref{quasihomogenized}, we get for $\varphi=u^\delta-u^*$ and $p>2$ that 
		\begin{eqnarray*}
			\|(u^\delta)'-(u^*)'\|_{L^p_\delta(0,1)}^p&\leq& c\int_0^1q^\delta \left[a_p\left((u^\delta)'\right)-a_p\left((u^*)'\right)\right]\left[(u^\delta)'-(u^*)'\right]dx\\
			&=&c\int_0^1 (\hat{f}-a_p(u^\delta))(u^\delta-u^*)dx-c\int_0^1q^\delta a_p\left((u^*)'\right)\left[(u^\delta)'-(u^*)'\right]dx\\
			&\to& 0.
		\end{eqnarray*}
		Hence, using the equivalence of norms, we get
		$$
		\|(u^\delta)'-(u^*)'\|_{L^p(0,1)}\leq \|(u^\delta)'-(u^*)'\|_{L^p_\delta(0,1)}\to0,
		$$
		as $\delta\to0$, which implies \eqref{convfinalthm}. Thus, we have that $u^* \in W^{1,p}(0,1)$ satisfies 
		\begin{equation}  \label{homogenizedlimitproof}
		\int_0^1 \left\lbrace q(x)|(u^*)'|^{p-2}(u^*)'\varphi' +r(x)\,|u^*|^{p-2}u^*\varphi \right\rbrace dx=\int_0^1\hat{f}\varphi dx,
		\end{equation}
		for all $\varphi\in W^{1,p}(0,1)$ and for $p\geq 2$. For $1<p<2$, we use similar arguments as in the proof of Theorem \ref{thmpiecewise} when we obtained \eqref{correctorthm0001}.

		To finish the proof, we need to show that $u^*=u_0$ in $(0,1)$ where $u_0$ is given in \eqref{1160}.
		
		Let $\eta$ be a positive small number and let $\varphi\in C_0^\infty(0,1)$. Notice that
		\begin{equation}\label{ineqthmfinal}
		\begin{gathered}
		\int_0^1(u_0-u^*)\varphi dx=\int_0^1\left(u_0-\dfrac{L}{|Y^*(x)|\varepsilon}\int_0^{\varepsilon G_\varepsilon(x)} u_\varepsilon(x,y) dy\right)\varphi(x) dx\\
		+\int_0^1\left(\dfrac{L}{|Y^*(x)|\varepsilon}\int_0^{\varepsilon G_\varepsilon(x)} u_\varepsilon(x,y) -P_{1+\delta/G_0}u_{\varepsilon,\delta}(x,y)dy\right)\varphi(x) dx\\
		+\int_0^1\left(\dfrac{L}{|Y^*(x)|\varepsilon}\int_0^{\varepsilon G_\varepsilon(x)} P_{1+\delta/G_0}u_{\varepsilon,\delta}(x,y)-u^
		\delta(x) dy\right)\varphi(x) dx\\
		+\int_0^1\left(\dfrac{L}{|Y^*(x)|\varepsilon}\int_0^{\varepsilon G_\varepsilon(x)} u^
		\delta(x)-u^*(x) dy\right)\varphi(x) dx,
		\end{gathered}
		\end{equation}
		where $P_{1+\delta/G_0}$ is the operator defined in  \eqref{operatorcomparison}.
		
		Now, due to definition \eqref{operatorcomparison}, notation \eqref{rescaleddomain} and an appropriated  change of variables, we get 
		\begin{equation*}
		\begin{gathered}
		\int_0^1\left(\dfrac{L}{\varepsilon}\int_0^{\varepsilon G_\varepsilon(x)} P_{1+\delta/G_0}u_{\varepsilon,\delta}(x,y)-u^
		\delta(x) dy\right)\varphi(x) dx\leq c |||P_{1+\delta/G_0}u_{\varepsilon,\delta}-u^
		\delta|||_{L^p(R^\varepsilon)}\\
		\leq c|||P_{1+\delta/G_0}u_{\varepsilon,\delta}-u^
		\delta|||_{L^p(R^{\varepsilon,\delta}(1+\delta))}= c|||u_{\varepsilon,\delta}-u^
		\delta|||_{L^p(R^{\varepsilon,\delta})}.
		\end{gathered}
		\end{equation*}
		Thus, we can rewrite \eqref{ineqthmfinal} as
		$$
		\begin{gathered}
		\left|\int_0^1(u_0-u^*)\varphi dx\right|\leq\left|\int_0^1\left(u_0-\dfrac{L}{\varepsilon}\int_0^{\varepsilon G_\varepsilon(x)} u_\varepsilon(x,y) dy\right)\varphi(x) dx\right|\\
		+c|||u_\varepsilon-P_{1+\delta/G_0}u_{\varepsilon,\delta}|||_{L^p(R^\varepsilon)}+c|||u_{\varepsilon,\delta}-u^
		\delta|||_{L^p(R^{\varepsilon,\delta})}+c\|u^\delta-u^*\|_{L^p(0,1)}.
		\end{gathered}
		$$
		
		From \eqref{convepsilondelta} and Remark \ref{corollarycomparison}, we can take $\delta>0$ small enough such that $|||u_\varepsilon-P_{1+\delta/G_0}u_{\varepsilon,\delta}|||_{L^p(R^\varepsilon)}\leq\eta$ and $|||u_{\varepsilon,\delta}-u^\delta|||_{L^p(R^{\varepsilon,\delta})}\leq \eta$ uniformly in $\varepsilon>0$.
		Also, from \eqref{convfinalthm}, we can choose $\varepsilon_1>0$ such that $|||u^*-u^\delta|||_{L^p(0,1)}\leq\eta$ for $0<\varepsilon<\varepsilon_1$. 
		
		Moreover, from \eqref{1160} and Proposition \ref{propLPunfolfingconv}, we have
		$$
		\int_0^1\left(u_0-\dfrac{L}{|Y^*(x)|\varepsilon}\int_0^{\varepsilon G_\varepsilon(x)} u_\varepsilon(x,y) dy\right)\varphi(x) dx\to 0,\quad\mbox{as}\quad \varepsilon\to 0. 
		$$
		Therefore, there exists $\varepsilon_2>0$ such that 
		$$
		\left|\int_0^1\left(u_0-\dfrac{L}{|Y^*(x)|\varepsilon}\int_0^{\varepsilon G_\varepsilon(x)} u_\varepsilon(x,y) dy\right)\varphi(x) dx\right|\leq \eta
		$$ 
		whenever $0<\varepsilon<\varepsilon_2$. Hence, setting $\varepsilon=\min\{\varepsilon_1,\varepsilon_2 \}$ we get 
		$$
		\left|\int_0^1(u_0-u^*)\varphi dx \right|\leq 4\eta
		$$ 
		Since $\varphi$ and $\eta$ are arbitrarily, we conclude that $u^*=u_0$. 
	\end{proof}

\section{Appendix}

In the proof of the main result, we used $q^\delta\to q$ uniformly to obtain \eqref{homogenizedlimitproof}. Recall that $q^\delta$ and $q$ are given by \eqref{pqdefinition} and \eqref{pqdefinitionlimit} respectively. Here we prove such convergence. For this sake, let us first set 
\begin{equation}\label{setAM}
A(M)=\left\lbrace G\in C^1(\mathbb{R}):\, G \mbox{ is }L-\mbox{periodic, }0<G_0\leq G(\cdot)\leq G_1 \textrm{ with } |G'(s)|\leq M\right\rbrace.
\end{equation}
Hence, for any $\bar G \in A(M)$, we can consider the  problem
\begin{equation}\label{problemNappendix}
\int_{Y^*_{\bar G}}|\nabla\bar v|^{p-2}\nabla\bar v\nabla\varphi dy_1dy_2=0,\quad\forall\varphi\in  W^{1,p}_{\#,0}(Y^*_{\bar G}) 
\end{equation}
where $W^{1,p}_{\#,0}(Y^*_{\bar G})$ is the space of functions $W^{1,p}_{\#}(Y^*_{\bar G})$ with zero average, 
$$
Y^*_{\bar G}=\left\lbrace(y_1,y_2)\in\mathbb{R}^2:0<y_1<L,0<y_2<\bar G(y_2)\right\rbrace
$$
and we are looking for solutions $\bar v$ such that $(\bar v-y_1)\in W^{1,p}_{\#,0}(Y^*_{\bar G})$.

Now, for any $\bar G$, $G \in A(M)$, let us consider the following transformation
\begin{eqnarray*}
	\begin{array}{ccccl}
		L&:&Y^*_{G}&\mapsto& Y^*_{\bar G}\\
		&&(z_1,z_2)&\to &(z_1,F(z_1)z_2)=(y_1,y_2)
	\end{array}
\end{eqnarray*}
where $$F=\dfrac{\bar G}{G}.$$ 

The Jacobian matrix for $L$ is 
$$
JL(z_1,z_2)=\left(\begin{array}{cc}
1&0\\
F'(z_1)z_2&F(z_1)
\end{array}\right)
$$
with ${\rm det}(JL) = F$. 
Also, we can consider
$$
\begin{gathered}
\mathcal{L} \nabla U =\left(\begin{array}{cc}
1 & -\dfrac{F'}{F}z_2\\
0 & 1/F
\end{array}\right)\nabla U= \left(\partial_{z_1}U-\dfrac{F'}{F}z_2\partial_{z_2}U,\dfrac{1}{F}\partial_{z_2}U\right) \quad \textrm{ and } 
\\\mathcal{B}\nabla  U=\left(\partial_{z_1}U+\dfrac{F'z_2}{F}\partial_{z_2}U,- \dfrac{F'z_2}{F}\partial_{z_1}U +\dfrac{1}{F^{2}}\left[1+(z_2 F')^2\right]\partial_{z_2}U\right).
\end{gathered}
$$
It is not difficult to see that $\mathcal{B}=\mathcal{L}^T\mathcal{L}$.

Then, we can use the change of variables given by $L$ to rewrite \eqref{problemNappendix} in the region $Y^*_{G}$ as 
\begin{equation}\label{problemNappendixchangedvariational}
\int_{Y^*_G}|\mathcal{L}\nabla \bar v|^{p-2}\mathcal{L}\nabla \bar v \mathcal{L}\nabla \left(\dfrac{\varphi}{F}\right) F\, dz_1dz_2=0\,, \forall \varphi\in W^{1,p}_{\#,0}(Y^*_{G}).
\end{equation}
Notice that this problem still has unique solution $\bar v\in W^{1,p}(Y^*_G)$ with $(\bar v-z_1)\in W^{1,p}_{\#,0}(Y^*_{G})$ by Minty-Browder's Theorem.

By the coercivity of \eqref{problemNappendixchangedvariational}, we get
$$
\begin{gathered}
\|\nabla \bar v\|_{L^p(Y^*_G)}^p\leq \int_{Y^*_G}|\mathcal{L}\nabla \bar v|^{p-2}\mathcal{L}\nabla \bar v \mathcal{L}\nabla \left(\dfrac{\bar v}{F}\right) F\, dz_1dz_2\\
=-\int_{Y^*_G}|\mathcal{L}\nabla \bar v|^{p-2}\mathcal{L}\nabla \bar v \mathcal{L}\nabla \left(\dfrac{z_1}{F}\right) F\, dz_1dz_2\\
\leq c \|\mathcal{L}\nabla \bar v\|_{L^p(Y^*_G)}^{p-1} \leq c\|\nabla \bar v\|_{L^p(Y^*_G)}^{p-1},
\end{gathered}
$$
which means that the solutions are uniformly bounded by a constant independent on $\bar G$ and $G$. 

Now, let us compare the solutions of \eqref{problemNappendix} for $\bar G = G$ and \eqref{problemNappendixchangedvariational}. We need to analyze
\begin{equation}\label{eqmain}
\begin{gathered}
\int_{Y^*_G}\left[|\mathcal{L}\nabla \bar v|^{p-2}\mathcal{L}\nabla \bar v-|\nabla v|^{p-2}\nabla v\right](\mathcal{L}\nabla \bar v-\nabla v) dz_1dz_2\\
=\int_{Y^*_G}\left[|\mathcal{L}\nabla \bar v|^{p-2}\mathcal{L}\nabla \bar v-|\nabla v|^{p-2}\nabla v\right](\mathcal{L}\nabla \bar v-(1,0)+(1,0)-\nabla v) dz_1dz_2.
\end{gathered}
\end{equation}
Notice that $\mathcal{L}(1,0)=(1,0)$.
We will distribute the terms finding estimative for each one.

First, observe that for any test function $\varphi\in W^{1,p}_{\#,0}(Y^*_G)$ in \eqref{problemNappendixchangedvariational}, we have
\begin{equation}\label{eq001}
\int_{Y^*_G}|\mathcal{L}\nabla \bar v|^{p-2}\mathcal{L}\nabla \bar v\mathcal{L}\nabla \varphi dz_1dz_2 
= \int_{Y^*_G}|\mathcal{L}\nabla \bar v|^{p-2}\mathcal{L}\nabla \bar v\varphi\left(\dfrac{F'}{F},0\right) dz_1dz_2.
\end{equation}
Now, take $\varphi=(\bar v-z_1)$ in \eqref{eq001}. Then,
\begin{equation}\label{eq01}
\begin{gathered}
\int_{Y^*_G}|\mathcal{L}\nabla \bar v|^{p-2}\mathcal{L}\nabla \bar v\mathcal{L}\nabla (\bar v-z_1) dz_1dz_2=\int_{Y^*_G}|\mathcal{L}\nabla \bar v|^{p-2}\mathcal{L}\nabla \bar v(\bar v-z_1)\left(\dfrac{F'}{F},0\right) dz_1dz_2
\end{gathered}
\end{equation}

On the other side, we can compute 
\begin{equation}\label{eq02}
\begin{gathered}
\int_{Y^*_G}|\mathcal{L}\nabla \bar v|^{p-2}\mathcal{L}\nabla \bar v ((1,0)-\nabla v) dz_1dz_2\\=\int_{Y^*_G}|\mathcal{L}\nabla \bar v|^{p-2}\mathcal{L}\nabla \bar v ((1,0)-\nabla v+\mathcal{L}\nabla v-(1,0)+(1,0)-\mathcal{L}\nabla v) dz_1dz_2\\
=\int_{Y^*_G}|\mathcal{L}\nabla \bar v|^{p-2}\mathcal{L}\nabla \bar v (-\nabla v+\mathcal{L}\nabla v) dz_1dz_2+\int_{Y^*_G}|\mathcal{L}\nabla \bar v|^{p-2}\mathcal{L}\nabla \bar v \mathcal{L}\nabla (z_1 - v) dz_1dz_2\\
=-\int_{Y^*_G}|\mathcal{L}\nabla \bar v|^{p-2}\mathcal{L}\nabla \bar v (\mathcal{L}-I)\nabla v dz_1dz_2+\int_{Y^*_G}|\mathcal{L}\nabla \bar v|^{p-2}\mathcal{L}\nabla \bar v(z_1-v)\left(\dfrac{F'}{F},0\right) dz_1dz_2
\end{gathered}
\end{equation}
by \eqref{eq001} with $\varphi=(z_1-v)$. 

Next, take $(\bar v -z_1)\in W^{1,p}_{\#,0}(Y^*_G)$ as a test function in \eqref{problemNappendix}. Then, 
\begin{equation}\label{eq03}
\int_{Y^*_G}|\nabla v|^{p-2}\nabla v(\nabla\bar v -(1,0)) dz_1dz_2=0.
\end{equation}
Finally, due to \eqref{eq03}, we have 
\begin{equation}\label{eq04}
\begin{gathered}
\int_{Y^*_G}|\nabla v|^{p-2}\nabla v(\mathcal{L}\nabla \bar v -(1,0)) dz_1dz_2\\=\int_{Y^*_G}|\nabla v|^{p-2}\nabla v(\mathcal{L}\nabla \bar v -(1,0)) dz_1dz_2- \int_{Y^*_G}|\nabla v|^{p-2}\nabla v(\nabla \bar v -(1,0)) dz_1dz_2\\
=\int_{Y^*_G}|\nabla v|^{p-2}\nabla v(\mathcal{L}-I)\nabla\bar v dz_1dz_2.
\end{gathered}
\end{equation}

Hence, putting together \eqref{eqmain}, \eqref{eq01}, \eqref{eq02}, \eqref{eq03} and \eqref{eq04}, we obtain
\begin{equation}\label{eq002}
\begin{gathered}
\int_{Y^*_G}\left[|\mathcal{L}\nabla \bar v|^{p-2}\mathcal{L}\nabla \bar v-|\nabla v|^{p-2}\nabla v\right](\mathcal{L}\nabla \bar v-\nabla v) dz_1dz_2
\\=\int_{Y^*_G}|\mathcal{L}\nabla \bar v|^{p-2}\mathcal{L}\nabla \bar v(\bar v-z_1)\left(\dfrac{F'}{F},0\right) dz_1dz_2\\
-\int_{Y^*_G}|\mathcal{L}\nabla \bar v|^{p-2}\mathcal{L}\nabla \bar v (\mathcal{L}-I)\nabla v dz_1dz_2+\int_{Y^*_G}|\mathcal{L}\nabla \bar v|^{p-2}\mathcal{L}\nabla \bar v(z_1-v)\left(\dfrac{F'}{F},0\right) dz_1dz_2\\
- \int_{Y^*_G}|\nabla v|^{p-2}\nabla v(\mathcal{L}-I)\nabla\bar v dz_1dz_2.
\end{gathered}
\end{equation}

Now, one can apply H\"older and Poincar\'e-Wirtinger's inequalities in \eqref{eq002} to obtain
\begin{equation}\label{eq003}
\begin{gathered}
\int_{Y^*_G}\left[|\mathcal{L}\nabla \bar v|^{p-2}\mathcal{L}\nabla \bar v-|\nabla v|^{p-2}\nabla v\right](\mathcal{L}\nabla \bar v-\nabla v) dz_1dz_2\\
\leq \|\mathcal{L}\nabla \bar v\|_{L^p(Y^*_G)}^{p-1}\|\nabla \bar v\|_{L^p(Y^*_G)}\left\|\dfrac{F'}{F}\right\|_{L^\infty }+\|\mathcal{L}\nabla \bar v\|_{L^p(Y^*_G)}^{p-1}\left\|\mathcal{L}-I\right\|_{L^\infty }\|\nabla v \|_{L^p(Y^*_G)}\\
+\|\mathcal{L}\nabla \bar v\|_{L^p(Y^*_G)}^{p-1}\|\nabla   v\|_{L^p(Y^*_G)}\left\|\dfrac{F'}{F}\right\|_{L^\infty }+\|\nabla  v\|_{L^p(Y^*_G)}^{p-1}\left\|\mathcal{L}-I\right\|_{L^\infty }\|\nabla\bar v \|_{L^p(Y^*_G)}.
\end{gathered}
\end{equation}

Note that
\begin{equation}\label{estimF}
\begin{gathered}
\left\|\dfrac{F'}{F}\right\|_{L^\infty}\leq c \|\bar{G}-G\|_{C^1}\mbox{ and }\|\mathcal{L}-I\|_{L^\infty}\leq c \|\bar{G}-G\|_{C^1}.
\end{gathered}
\end{equation}
Also, $\|\nabla  v\|_{L^p(Y^*_G)}$, $\|\nabla \bar v\|_{L^p(Y^*_G)}$, $\|\mathcal{L}\nabla \bar v\|_{L^p(Y^*_G)}$ and $\|\mathcal{L}\nabla v\|_{L^p(Y^*_G)}$ are uniformly bounded. Thus, by \eqref{eq003}
\begin{equation}\label{ineqapp}
\int_{Y^*_G}\left[|\mathcal{L}\nabla \bar v|^{p-2}\mathcal{L}\nabla \bar v-|\nabla v|^{p-2}\nabla v\right](\mathcal{L}\nabla \bar v-\nabla v) dz_1dz_2\leq c \|\bar{G}-G\|_{C^1}.
\end{equation}

If $p\geq 2$, we get from Proposition \ref{proposicaoplaplaciano} and \eqref{ineqapp} that 
\begin{equation*} 
\begin{gathered}
\|\mathcal{L}\nabla \bar v-\nabla v\|^p_{L^p(Y^*_G)}\leq c\int_{Y^*_G}\left[|\mathcal{L}\nabla \bar v|^{p-2}\mathcal{L}\nabla \bar v-|\nabla v|^{p-2}\nabla v\right](\mathcal{L}\nabla \bar v-\nabla v) dz_1dz_2\\
\leq c \|\bar{G}-G\|_{C^1}.
\end{gathered}
\end{equation*}
On the other side, if $1<p<2$, we get from H\"older's inequality, Proposition \ref{proposicaoplaplaciano} and \eqref{ineqapp}, that
\begin{equation*}
\begin{gathered}
\|\mathcal{L}\nabla \bar v-\nabla v\|^p_{L^p(Y^*_G)}\leq c\left\lbrace\int_{Y^*_G}\left[|\mathcal{L}\nabla \bar v|^{p-2}\mathcal{L}\nabla \bar v-|\nabla v|^{p-2}\nabla v\right](\mathcal{L}\nabla \bar v-\nabla v) dz_1dz_2\right\rbrace^{p/2}\\
\qquad\qquad\qquad\qquad\qquad\left[\int_{Y^*_G}(1+|\mathcal{L}\nabla \bar{v}|+|\nabla v|)^p\right]^{(2-p)/2}\\
\leq c \|\bar{G}-G\|_{C^1}^{p/2},
\end{gathered}
\end{equation*}

Therefore, for $1<p<\infty$, we have 
\begin{equation}\label{ineqmainapp}
\|\mathcal{L}\nabla \bar v-\nabla v\|_{L^p(Y^*_G)}\leq c \|\bar{G}-G\|_{C^1}^{\alpha}
\end{equation}
where $\alpha=1/2$ if $1<p<2$ and $\alpha=1/p$ if $p\geq 2$.

Finally, since
\begin{equation*}
\begin{gathered}
\|\nabla \bar{v}-\nabla v\|_{L^p(Y^*_G)}\leq\|\mathcal{L}\nabla \bar v-\nabla \bar v\|_{L^p(Y^*_G)}+ \|\mathcal{L}\nabla \bar v-\nabla v\|_{L^p(Y^*_G)},
\end{gathered}
\end{equation*}
we conclude by \eqref{ineqmainapp} and \eqref{estimF} that
\begin{equation*}
\begin{gathered}
\|\nabla \bar{v}-\nabla v\|_{L^p(Y^*_G)}
\leq c \|\bar{G}-G\|_{C^1}+c \|\bar{G}-G\|_{C^1}^{\alpha}.
\end{gathered}
\end{equation*}
We have the following lemma:

\begin{lemma}\label{lemmaappendix}
	Let us consider the family of admissible functions $G\in A(M)$ for some constant $M>0$ where $A(M)$ is defined by \eqref{setAM}.
	
	Then, for each $\varepsilon>0$, there exists $\delta>0$ such that if $G,\bar{G}\in A(M)$ with $\|\bar{G}-G\|\leq \delta$, then
	$$
	\|\nabla \bar{v}-\nabla v\|_{L^p(Y^*_G)}\leq c(\varepsilon+\varepsilon^{\alpha}),
	$$
	where $\alpha=1/2$ if $1<p<2$ and $\alpha=1/p$ if $p\geq 2$ and $c$ is a constant which depends only on $p,G_{0},G_{1}$. 
	In particular, we have that 
	$$
	|q(\bar{G})-q(G)|\leq c(\varepsilon+\varepsilon^{\alpha}),
	$$
	where
	$$
	q(\bar G)=\int_{Y^*_{\bar{G}}}|\nabla \bar{v}|^{p-2}\partial_{y_1}\bar{v} dy_1dy_2
	$$
	and $\bar{v}$ is the solution of \eqref{problemNappendix} in the region $Y^*_{\bar{G}}$ set by $\bar G$.
\end{lemma}


\begin{thebibliography}{BH}
	
	
	
	
	\bibitem{arrieta2011}
	J.~M. Arrieta, A.~N. Carvalho, M.~C. Pereira, and R.~P. Silva. \emph{Semilinear
		parabolic problems in thin domains with a highly oscillatory boundary}. Nonlinear Analysis 74-15 
	(2011) 5111--5132.
	
	\bibitem{nakasato1} J. M. Arrieta, J. C. Nakasato and M. C. Pereira. \emph{The p-laplacian equation in thin domains: The unfolding approach}. Submitted.
	
	\bibitem{AP10} J. M. Arrieta and M. C. Pereira. \emph{Homogenization in a thin domain with an oscillatory boundary}. J. Math. Pures et Appl. 96 (2011) 29--57.
	
	
	\bibitem{AM} J. M. Arrieta and M. Villanueva-Pesqueira, \emph{Unfolding operator method for thin domains with a locally periodic highly oscillatory boundary}. SIAM J. of Math. Analysis 48-3 (2016) 1634--1671.
	
	\bibitem{AM2} J. M. Arrieta and M. Villanueva-Pesqueira. \emph{Thin domains with non-smooth oscillatory boundaries}. J. of Math. Analysis and Appl. 446-1 (2017) 130--164. 
	
	
	
	
	\bibitem{BPazanin} M. Benes and I. Pazanin, \emph{Effective flow of incompressible micropolar fluid through a system of thin pipes}. Acta Applicandae Mathematicae 143 (1) (2016) 29--43.
	
	
	
	
	
	
	
	
	
	
	
	
	
	\bibitem{GH} A. Gaudiello, K. Hamdache, {\it A reduced model for the polarization in a ferroelectric thin wire}. NoDEA Nonlinear Differential Equations Appl. 22 (6) (2015) 1883--1896.
	
	\bibitem{HaleRaugel} J. K. Hale and G. Raugel, {\it Reaction-diffusion equations on thin domains}. J. Math. Pures et Appl. 9 (71) (1992) 33--95.
	
	\bibitem{Jimbo} S. Jimbo, H. Kozono, Y. Teramoto and E. Ushikoshi. \emph{Hadamard variational formula for eigenvalues of the Stokes equations and its application}. Math. Ann. 368 (2017) 877--884.
	
	
	
	\bibitem{lindqvist} P. Lindqvist. Notes on the p-Laplace equation. University of Jyväskylä, 2017.
	
	
	
	
	
	
	
	
	
	\bibitem{AJM} A. Nogueira, J. C. Nakasato and M. C. Pereira. {\it Concentrated reaction terms on the boundary of rough domains for a quasilinear equation}. Appl. Math. Letters 102 (2020) 106120.
	
	
	
	
	\bibitem{MZAMP} M. C. Pereira. \emph{Asymptotic analysis of a semilinear elliptic equation in highly oscillating thin domains}. Zeitschrift fur Angewandte Mathematik und Physik, 67 (2016) 1--14.
	
	
	
	
	\bibitem{MRi2} M. C. Pereira and R. P. Silva. \emph{Remarks on the p-laplacian on thin domains}. Progress in Nonlinear Diff. Eq. and Their Appl.(2015) 389--403.
	
	
	
	
	\bibitem{Villanueva2016} M.Villanueva-Pesqueira, \emph{ Homogenization of Elliptic problems in thin domains with oscillatory boundaries}, Ph.D. Thesis, Universidad Complutense de Madrid, 2016.
	
	
\end{thebibliography}
\end{document}